\documentclass{amsart}
\usepackage{amsmath}
\usepackage{amsthm}
\usepackage{graphicx}
\usepackage{color}
\usepackage{amssymb}
\usepackage{mathrsfs}
\usepackage{url}
\usepackage{mathtools}

\newcommand{\re}{\mathbb{R}}
\newcommand{\cpx}{\mathbb{C}}

\newcommand{\N}{\mathbb{N}}

\renewcommand{\P}{\mathbb{P}}

\newcommand{\diag}{\mbox{diag}}

\newcommand{\lmd}{\lambda}

\newcommand{\eps}{\epsilon}

\newcommand{\dt}{\delta}
\newcommand{\Dt}{\Delta}
\newcommand{\tx}{\tilde{x}}
\newcommand{\ty}{\tilde{y}}
\def\bx{\bar{x}}
\def\by{\bar{y}}
\newcommand{\tg}{\tilde{g}}
\renewcommand{\th}{\tilde{h}}
\def\af{\alpha}

\def\rank{\mbox{rank}}

\newcommand{\sig}{\sigma}
\newcommand{\Sig}{\Sigma}

\newcommand{\reff}[1]{(\ref{#1})}
\newcommand{\pt}{\partial}
\newcommand{\prm}{\prime}

\newcommand{\mc}[1]{\mathcal{#1}}

\newcommand{\cv}[1]{\mbox{conv}(#1)}

\newcommand{\bdes}{\begin{description}}
\newcommand{\edes}{\end{description}}

\newcommand{\bal}{\begin{align}}
\newcommand{\eal}{\end{align}}

\newcommand{\bnum}{\begin{enumerate}}
\newcommand{\enum}{\end{enumerate}}

\newcommand{\bit}{\begin{itemize}}
\newcommand{\eit}{\end{itemize}}

\newcommand{\bea}{\begin{eqnarray}}
\newcommand{\eea}{\end{eqnarray}}
\newcommand{\be}{\begin{equation}}
\newcommand{\ee}{\end{equation}}

\newcommand{\baray}{\begin{array}}
\newcommand{\earay}{\end{array}}

\newcommand{\bsry}{\begin{subarray}}
\newcommand{\esry}{\end{subarray}}

\newcommand{\bca}{\begin{cases}}
\newcommand{\eca}{\end{cases}}

\newcommand{\bcen}{\begin{center}}
\newcommand{\ecen}{\end{center}}

\newcommand{\bbm}{\begin{bmatrix}}
\newcommand{\ebm}{\end{bmatrix}}

\newcommand{\bmx}{\begin{matrix}}
\newcommand{\emx}{\end{matrix}}

\newcommand{\bpm}{\begin{pmatrix}}
\newcommand{\epm}{\end{pmatrix}}

\newcommand{\btab}{\begin{tabular}}
\newcommand{\etab}{\end{tabular}}

\newtheorem{theorem}{Theorem}[section]

\newtheorem{prop}[theorem]{Proposition}

\newtheorem{cond}[theorem]{Condition}
\theoremstyle{definition}

\newtheorem{example}[theorem]{Example}
\newtheorem{exm}[theorem]{Example}
\newtheorem{question}[theorem]{Question}
\newtheorem{alg}[theorem]{Algorithm}

\numberwithin{equation}{section}

\begin{document}

\title[The Saddle Point Problem of Polynomials]
{The Saddle Point Problem of Polynomials}

\author{Jiawang Nie, Zi Yang}
\address{Department of Mathematics,
University of California San Diego,
9500 Gilman Drive, La Jolla, CA, USA, 92093.}
\email{njw@math.ucsd.edu, ziy109@ucsd.edu}

\author{Guangming Zhou}
\address{School of Mathematics and Computational Science
Xiangtan University,
Xiangtan, Hunan, 411105, China.}
\email{zhougm@xtu.edu.cn}

\subjclass[2010]{90C22,90C47,49K35,65K05}

\date{}

\keywords{saddle point, polynomial, nonsingularity,
Lasserre relaxation, semidefinite program}

\begin{abstract}
This paper studies the saddle point problem of polynomials.
We give an algorithm for computing saddle points.
It is based on solving Lasserre's hierarchy of semidefinite relaxations.
Under some genericity assumptions on defining polynomials,
we show that: i) if there exists a saddle point,
our algorithm can get one by solving a finite hierarchy of
Lasserre type semidefinite relaxations;
ii) if there is no saddle point,
our algorithm can detect its nonexistence. 
\end{abstract}

\maketitle

\section{Introduction}

Let $X \subseteq \re^n,  Y \subseteq \re^m$ be two sets (for dimensions $n,m>0$),
and let $F(x,y)$ be a continuous function in
$(x,y) \in X \times Y$.
A pair $(x^*,y^*) \in X \times Y$ is said to be a saddle point
of $F(x,y)$ over $X \times Y$ if
\be \label{dfsad:F:XY}
F(x^*,y) \, \le \, F(x^*,y^*)  \,\le \, F(x,y^*) \quad
\forall \, x \in X, \, \forall \, y \in Y.
\ee
The above implies that
\[
 F(x^*,y^*) \,=\,  \min_{x\in X} F(x,y^*)
\leq \max_{y\in Y} \min_{x\in X} F(x,y),
\]
\[
F(x^*,y^*) \,=\,  \max_{y\in Y} F(x^*,y)
\geq \min_{x\in X} \max_{y\in Y} F(x,y).
\]
On the other hand, it always holds that
\[
 \max_{y\in Y}\min_{x\in X} F(x,y) \, \le \, \min_{x\in X}\max_{y\in Y} F(x,y).
\]
Therefore, if $(x^*,y^*)$ is a saddle point, then
\be \label{maxmin=minmax}
 \min_{x\in X}\max_{y\in Y} F(x,y) \,=\, F(x^*,y^*)
 \, = \, \max_{y\in Y}\min_{x\in X} F(x,y).
\ee
All saddle points share the same objective value,
although there may exist different saddle points.
The definition of saddle points in \reff{dfsad:F:XY} requires the inequalities
to hold for all points in the feasible sets $X,Y$.
That is, when $y$ is fixed to $y^*$, $x^*$ is a global minimizer of
$F(x,y^*)$ over $X$; when $x$ is fixed to $x^*$, $y^*$ is a global maximizer of
$F(x^*,y)$ over $Y$.
Certainly, $x^*$ must also be a local minimizer of $F(x,y^*)$
and $y^*$ must be a local maximizer of $F(x^*,y)$.
So, the local optimality conditions can be applied at $(x^*,y^*)$.
However, they are not sufficient to guarantee that $(x^*,y^*)$ is a saddle point,
since \reff{dfsad:F:XY} needs to be satisfied for all feasible points.

The saddle point problem of polynomials (SPPP) is for cases that $F(x,y)$
is a polynomial function in $(x,y)$ and $X,Y$ are semialgebraic sets, i.e.,
they are described by polynomial equalities and/or inequalities.
The SPPP concerns the existence of saddle points and
the computation of them if they exist.
When $F$ is convex-concave in $(x,y)$ and $X,Y$
are nonempty compact convex sets,
there exists a saddle point. We refer to \cite[\S2.6]{BNO03}
for the classical theory for convex-concave type saddle point problems.
The SPPPs have broad applications.
They are of fundamental importance
in duality theory for constrained optimization,
min-max optimization and game theory \cite{BNO03,LarLas12,ShahPar07}.
The following are some applications.
\begin{itemize}

\item  Zero sum games can be formulated as
saddle point problems \cite{bac98,Nash,Rat13}.
In a zero sum game with two players,
suppose the first player has the strategy vector $x :=(x_1, \ldots, x_n)$
and the second player has the strategy vector $y :=(y_1, \ldots, y_m)$.
The strategies $x,y$ usually represent probability measures over finite sets,
for instance, $x \in \Delta_n$, $y \in \Delta_m$.
(The notation $\Delta_n$ denotes the standard simplex in $\re^n$.)
A typical profit function of the first player is
\[
	f_1(x,y) \, := \, x^TA_1x + y^TA_2y + x^TBy,
\]
for matrices $A_1,A_2, B$.
For the zero sum game, the profit function
$f_2(x,y)$ of the second player is $-f_1(x,y)$.
Each player wants to maximize the profit,
for the given strategy of the other player.
A Nash equilibrium is a point $(x^*,y^*)$ such that the function $f_1(x,y^*)$ in $x$
achieves the maximum value at $x^*$, while the function $f_2(x^*,y)$ in $y$
achieves the maximum value at $y^*$. Since $f_1(x,y) + f_2(x,y) =0$,
the Nash equilibrium $(x^*,y^*)$ is a saddle point of the function $F := -f_1$
over feasible sets $X = \Dt_n$, $Y = \Dt_m$.

\item The image restoration \cite{HeYuan12} can be formulated
as a saddle point problem with the function
\[
 F(x,y) \, := \, x^TAy + \frac{1}{2}\|Bx-z\|^2
\]
and some feasible sets $X, Y$, for two given matrices $A,B$
and a given vector $z$. Here, the notation $\| \cdot \|$ denotes the Euclidean norm.
We refer to \cite{Anton11,Esser10,HeYuan12} for related work on this topic.

\item The saddle point problem plays an important role in
robust optimization \cite{bert10,hall03,Lasminmax11,zhu09}.
For instance, a statistical portfolio optimization problem is
\[
\min_{x\in X} \quad -\mu^Tx + x^TQx,
\]
where $Q$ is a covariance matrix,
$\mu$ is the estimation of some parameters, 
and $X$ is the feasible set for the decision variable $x$.
In applications, there often exists a perturbation for $(\mu,Q)$.
Suppose the perturbation for $(\mu, Q)$ is $(\dt \mu, \dt Q)$.
There are two types of robust optimization problems
\[
\baray{ccl}
\min\limits_{x\in X} & \max\limits_{(\delta \mu,\delta Q)\in Y} &
-(\mu+\delta \mu)^Tx +  x^T(Q+\delta Q)x, \\
\max\limits_{(\delta \mu,\delta Q)\in Y} & \min\limits_{x\in X} &
-(\mu+\delta \mu)^Tx +  x^T(Q+\delta Q)x,
\earay
\]
where $Y$ is the feasible set for the perturbation $(\dt \mu, \dt Q)$.
People are interested in $x^*$ and $(\dt \mu^*, \dt Q^*)$
such that the above two robust optimization problems
are solved simultaneously by them.
In view of \reff{maxmin=minmax}, this is equivalent to solving
a saddle point problem.

\end{itemize}

For convex-concave type saddle point problems, most existing methods are based on
gradients, subgradients, variational inequalities,
or other related techniques. For these classical methods,
we refer to the work by Chen, Lan and Ouyang~\cite{CLOu14},
Cox, Juditsky and Nemirovski \cite{CJN17},
He and Yuan \cite{HeYuan12}, He and Monteiro \cite{HeMon15},
Korpelevich~\cite{Kor77}, Maistroskii~\cite{Mai77},
Monteiro and Svaiter \cite{MonSva11}, Nemirovski \cite{Nem04},
Nedi\'{c} and Ozdaglar \cite{NedOzd}, and Zabotin~\cite{Zab88}.
For more general cases of non convex-concave type saddle point problems
(i.e., $F$ is not convex-concave, and/or one of the sets $X,Y$ is nonconvex),
the computational task for solving SPPPs is much harder.
A saddle point may, or may not, exist.
There is very little work for solving
non convex-concave saddle point problems \cite{DPGCGB,PDGB14}.
Obviously, SPPPs can be formulated as
a first order formula over the real field $\re$.
By the Tarski­-Seidenberg theorem \cite{BPR,Cavi12},
the SPPP is equivalent to a quantifier free formula.
Such quantifier free formula can be computed symbolically,
e.g., by cylindrical algebraic decompositions \cite{BPR,Cavi12,HED}.
Theoretically, the quantifier elimination (QE) method can solve SPPPs exactly,
but it typically has high computational complexity \cite{Brown07,Daven88,Regr92}.
Another straightforward approach for solving \reff{dfsad:F:XY}
is to compute all its critical points first and then select saddle points among them.
The complexity of computing critical points is given in \cite{Spa14}.
This approach typically has high computational cost, because the number of
critical points is dramatically high \cite{NR09}
and we need to check the global optimality relation in
\reff{dfsad:F:XY} for getting saddle points.
In the subsection~\ref{ssc:compa}, we will compare
the performance between these methods and the new one given in this paper.

The basic questions for saddle point problems are:
If a saddle point exists, how can we find it?
If it does not, how can we detect its nonexistence?
This paper discusses how to solve saddle point problems
that are given by polynomials and that are non convex-concave.
We give a numerical algorithm to solve SPPPs.

\subsection{Optimality conditions}
\label{ssc:opcd}

Throughout the paper, a property is said to hold {\it generically}
in a space if it is true everywhere except a subset of zero Lebesgue measure.
We refer to \cite{Har} for the notion of genericity in algebraic geometry.
Assume $X,Y$ are basic closed semialgebraic sets that are given as
\be \label{set:X}
X = \{ x \in \re^n \mid
g_i(x) = 0 \, (i \in \mc{E}^X_1), \,
g_i(x) \geq 0 \, (i \in \mc{E}^X_2) \},
\ee
\be \label{set:Y}
Y = \{ y \in \re^m \mid
h_j(y) = 0 \, (i \in \mc{E}^Y_1), \,
h_j(y) \geq 0 \, (i \in \mc{E}^Y_2) \}.
\ee
Here, each $g_i$ is a polynomial in $x := (x_1,\ldots,x_n)$
and each $h_j$ is a polynomial in $y := (y_1,\ldots, y_m)$.
The $\mc{E}^X_1, \mc{E}^X_2, \mc{E}^Y_1, \mc{E}^Y_2$
are disjoint labeling sets of finite cardinalities.
To distinguish equality and inequality constraints, denote the tuples
\be \label{gh:eq:in}
\boxed{
\baray{ll}
g_{eq} := (g_i)_{ i \in \mc{E}^X_1}, & h_{eq} := (h_j)_{ j \in \mc{E}^Y_1}, \\
g_{in} := (g_i)_{ i \in \mc{E}^X_2}, & h_{in} := (h_j)_{ j \in \mc{E}^Y_2}.
\earay
}
\ee
When $\mc{E}^X_1 = \emptyset$ (resp., $\mc{E}^X_2 = \emptyset$),
there is no equality (resp., inequality) constraint for $X$.
The same holds for $Y$. For convenience, denote the labeling sets
\[
\mc{E}^X:= \mc{E}^X_1 \cup  \mc{E}^X_2, \quad
\mc{E}^Y:= \mc{E}^Y_1 \cup  \mc{E}^Y_2.
\]
Suppose $(x^*, y^*)$ is a saddle point. Then, $x^*$ is a minimizer of
\be \label{minFx:y*}
\left\{ \baray{cl}
\min\limits_{x \in \re^n}  & F(x,y^*) \\
\mbox{subject to}  & g_i(x) = 0 \, (i \in \mc{E}^X_1), \\
 & g_i(x) \geq 0 \, (i \in \mc{E}^X_2),
\earay \right.
\ee
and $y^*$ is a maximizer of
\be \label{maxFy:x*}
\left\{ \baray{cl}
\max\limits_{y \in \re^m}  & F(x^*,y) \\
\mbox{subject to} & h_j(y) = 0 \, (j \in \mc{E}^Y_1), \\
 &  h_j(y) \geq 0 \, (j \in \mc{E}^Y_2).
\earay \right.
\ee
Under the linear independence constraint qualification (LICQ),
or other kinds of constraint qualifications (see \cite[\S3.3]{Brks}),
there exist Lagrange multipliers $\lmd_i, \mu_j$ such that
\be \label{opcd:F:x}
\nabla_x F(x^*, y^*) = \sum_{ i \in \mc{E}^X }  \lmd_i \nabla_x g_i(x^*) ,
\quad 0 \leq \lmd_i \perp g_i(x^*) \geq 0 \, (i \in \mc{E}^X_2),
\ee
\be \label{opcd:F:y}
\nabla_y F(x^*, y^*) = \sum_{j \in \mc{E}^Y } \mu_j  \nabla_y h_j(y^*),
\quad 0 \geq \mu_j \perp h_j(y^*) \geq  0 \, (j \in \mc{E}^Y_2).
\ee
In the above, $a\perp b$ means the product $a \cdot b =0$
and $\nabla_x F$ (resp., $\nabla_y F$) denotes
the gradient of $F(x,y)$ with respect to $x$ (resp., $y$).
When $g,h$ are nonsingular (see the below for the definition),
we can get explicit expressions for $\lmd_i, \mu_j$
in terms of $x^*,y^*$ (see \cite{LagExp}).
For convenience, write the labeling sets as
\[
\mc{E}^X =\{1,\ldots, \ell_1\}, \quad
\mc{E}^Y =\{1,\ldots, \ell_2\}.
\]
Then, the constraining polynomial tuples can be written as
\[
g \,  =  \,  (g_1, \ldots, g_{\ell_1}), \quad
h  \, =  \, (h_1, \ldots, h_{\ell_2}).
\]
The Lagrange multipliers can be written as vectors
\[
\lmd  \, = \,  (\lmd_1, \ldots, \lmd_{\ell_1}), \quad
\mu  \, =  \,  (\mu_1, \ldots, \mu_{\ell_2}).
\]
Denote the matrices
\be \label{G(x)}
G(x) \, := \,  \bbm
\nabla_x g_1(x) & \nabla_x g_2(x) &  \cdots &  \nabla_x g_{\ell_1}(x) \\
g_1(x) & 0  & \cdots & 0 \\
0  & g_2(x)  & \cdots & 0 \\
\vdots & \vdots & \ddots & \vdots \\
0  &  0  & \cdots & g_{\ell_1}(x)
\ebm,
\ee
\be \label{H(y)}
H(y) \, := \,  \bbm
\nabla_y h_1(y) & \nabla_y h_2(y) &  \cdots &  \nabla_y h_{\ell_2}(y) \\
h_1(y) & 0  & \cdots & 0 \\
0  & h_2(y)  & \cdots & 0 \\
\vdots & \vdots & \ddots & \vdots \\
0  &  0  & \cdots & h_{\ell_2}(y)
\ebm.
\ee
The tuple $g$ is said to be {\it nonsingular} if
$\rank \, G(x) = \ell_1$ for all $x \in \cpx^n$.
Similarly, $h$ is nonsingular if
$\rank \, H(y) = \ell_2$ for all $y \in \cpx^m$.
Note that if $g$ is nonsingular,
then LICQ must hold at $x^*$. Similarly, the LICQ holds at $y^*$
if $h$ is nonsingular.
When $g,h$ have generic coefficients (i.e., $g,h$ are generic),
the tuples $g,h$ are nonsingular.
The nonsingularity is a property that holds generically.
We refer to the work \cite{LagExp} for more details.

\subsection{Contributions}

This paper discusses how to solve saddle point problems of polynomials.
We assume that the sets $X, Y$ are given as in \reff{set:X}-\reff{set:Y}
and the defining polynomial tuples $g,h$
are nonsingular, i.e., the matrices
$G(x), H(y)$ have full column rank everywhere.
Then, as shown in \cite{LagExp},
there exist matrix polynomials $G_1(x), H_1(y)$ such that
($I_\ell$ denotes the $\ell \times \ell$ identity matrix)
\be \label{G1G=H1H=I}
G_1(x)G(x) \,=\, I_{\ell_1}, \quad
H_1(y)H(y) \,=\, I_{\ell_2}.
\ee
When $g,h$ have generic coefficients, they are nonsingular.
Clearly, the above and \reff{opcd:F:x}-\reff{opcd:F:y} imply that
\[
\lmd_i = G_1(x^*)_{i,1:n} \nabla_x F(x^*,y^*), \quad
\mu_j = H_1(y^*)_{j,1:m} \nabla_y F(x^*,y^*).
\]
(For a matrix $A$, the notation $A_{i,1:n}$ denotes its $i$th row
with column indices from $1$ to $n$.)
Denote the Lagrange polynomial tuples
\be \label{lmd(x,y)}
 \lmd(x,y):= G_1(x)_{:,1:n} \nabla_x F(x,y),
\ee
\be \label{mu(x,y)}
 \mu(x,y):= H_1(y)_{:,1:m} \nabla_y F(x,y).
\ee
(The notation $A_{:,1:n}$ denotes the submatrix of $A$
consisting of its first $n$ columns.)
At each saddle point $(x^*, y^*)$, the Lagrange multiplier vectors $\lmd, \mu$
in \reff{opcd:F:x}-\reff{opcd:F:y} can be expressed as
\[
\lmd = \lmd(x^*,y^*), \quad
\mu = \mu(x^*,y^*).
\]
Therefore, $(x^*,y^*)$ is a solution to the polynomial system
\be \label{SadKKT:XY}
\left\{ \baray{l}
g_i(x) = 0 \, (i \in \mc{E}^X_1), \,  h_j(y) =  0 \, (j \in \mc{E}^Y_1), \\
\nabla_xF(x, y) = \sum\limits_{ i \in \mc{E}^X }  \lmd_i(x,y) \nabla_x g_i(x) , \\
\nabla_yF(x, y) = \sum\limits_{j \in \mc{E}^Y } \mu_j(x,y)  \nabla_y h_j(y), \\
0 \leq \lmd_i(x,y) \perp g_i(x) \geq 0 \, (i \in \mc{E}^X_2), \\
0 \geq \mu_j(x,y) \perp h_j(y) \geq  0 \, (j \in \mc{E}^Y_2).
\earay\right.
\ee
However, not every solution $(x^*, y^*)$ to \reff{SadKKT:XY}
is a saddle point. This is because $x^*$
might not be a minimizer of \reff{minFx:y*}, and/or
$y^*$ might not be a maximizer of \reff{maxFy:x*}.
How can we use \reff{SadKKT:XY} to get a saddle point?
What further conditions do saddle points satisfy?
When a saddle point does not exist,
what is an appropriate certificate for the nonexistence?
This paper addresses these questions.
We give an algorithm for computing saddle points.
First, we compute a candidate saddle point $(x^*, y^*)$.
If it is verified to be a saddle point, then we are done.
If it is not, then either $x^*$ is not a minimizer of \reff{minFx:y*} or
$y^*$ is not a maximizer of \reff{maxFy:x*}.
For either case, we add a new valid constraint to
exclude such $(x^*, y^*)$, while all true saddle points
are not excluded. Then, we solve a new optimization problem,
together with the newly added constraints.
Repeating this process, we get an algorithm
(i.e., Algorithm~\ref{alg:pop:kt}) for solving SPPPs.
When the SPPP is given by generic polynomials,
we prove that Algorithm~\ref{alg:pop:kt} is able to
compute a saddle point if it exists,
and it can detect nonexistence if there does not exist one.
The candidate saddle points are optimizers of certain polynomial optimization problems.
We also show that these polynomial optimization problems can be solved exactly by
Lasserre's hierarchy of semidefinite relaxations,
under some genericity conditions on defining polynomials.
Since semidefinite programs are usually solved numerically
(e.g., by {\tt SeDuMi}) in practice,
the computed solutions are correct up to numerical errors.

The paper is organized as follows.
Section~\ref{sc:pre} reviews some basics for polynomial optimization.
Section~\ref{sc:algSPPP} gives an algorithm for solving SPPPs.
We prove its finite convergence when the polynomials are generic.
Section~\ref{sc:slvopt} discusses how to solve
the optimization problems that arise in Section~\ref{sc:algSPPP}.
Under some genericity conditions, we prove that
Lasserre type semidefinite relaxations
can solve those optimization problems exactly.
Proofs of some core theorems are given in Section~\ref{sc:proof}.
Numerical examples are given in Section~\ref{sc:num}.
Conclusions and some discussions are given in Section~\ref{sc:con}.

\section{Preliminaries}
\label{sc:pre}

This section reviews some basics in polynomial optimization.
We refer to \cite{BPT13,LasBk15,LasICM,LauSv,LauICM,Sch09}
for the books and surveys in this field.

\subsection{Notation}

The symbol $\N$ (resp., $\re$, $\cpx$) denotes the set of
nonnegative integral (resp., real, complex) numbers.
Denote by $\re[x] := \re[x_1,\ldots,x_n]$
the ring of polynomials in $x:=(x_1,\ldots,x_n)$
with real coefficients. The notation $\re[x]_d$
stands for the set of polynomials in $\re[x]$
with degrees $\leq d$. Sometimes, we need to work with polynomials
in $y:=(y_1,\ldots,y_m)$ or
$(x,y):=(x_1,\ldots,x_n,y_1,\ldots,y_m)$. The notation
$\re[y],\re[y]_d$, $\re[x,y],\re[x,y]_d$
are similarly defined.
For a polynomial $p$, $\deg(p)$ denotes its total degree.
For $t \in \re$, $\lceil t \rceil$ denotes
the smallest integer $\geq t$. For an integer $k>0$, denote
$[k] \, := \, \{1,2,\ldots, k\}.$
For $\af:= (\af_1, \ldots, \af_l) \in \N^l$ with an integer $l >0$,
denote $|\af|:=\af_1+\cdots+\af_l$.
For an integer $d > 0$, denote
\[
\N^l_d :=\{ \af  \in \N^l \mid |\af|  \leq d \}.
\]
For $z =(z_1,\ldots, z_l)$ and $\af = (\af_1, \ldots, \af_l)$, denote
\[
z^\af \, := \, z_1^{\af_1} \cdots z_l^{\af_l}, \quad
[z]_{d} := \bbm 1 & z_1 &\cdots & z_l & z_1^2 & z_1z_2 & \cdots & z_l^{d}\ebm^T.
\]
In particular, we often use the notation $[x]_d, [y]_d$ or $[(x,y)]_d$.
The superscript $^T$ denotes the transpose of a matrix/vector.
The notation $e_i$ denotes the $i$th standard unit vector,
while $e$ denotes the vector of all ones.
The notation $I_k$ denotes the $k$-by-$k$ identity matrix.
By writing $X\succeq 0$ (resp., $X\succ 0$), we mean that
$X$ is a symmetric positive semidefinite
(resp., positive definite) matrix.
For matrices $X_1,\ldots, X_r$, $\diag(X_1, \ldots, X_r)$
denotes the block diagonal matrix whose diagonal blocks
are $X_1,\ldots, X_r$. For a vector $z$,
$\|z\|$ denotes its standard Euclidean norm.
For a function $f$ in $x$, in $y$, or in $(x,y)$,
$\nabla_x f$ (resp., $\nabla_y f$) denotes
its gradient vector in $x$ (resp., in $y$).
In particular, $F_{x_i}$ denotes the partial derivative of
$F(x,y)$ with respect to $x_i$.

\subsection{Positive polynomials}
\label{ssc:posp}

In this subsection, we review some basic results
about positive polynomials in the ring $\re[x,y]$.
The same kind of results hold for
positive polynomials in $\re[x]$ or $\re[y]$.
An ideal $I$ of $\re[x,y]$ is a subset
such that $ I \cdot \re[x,y] \subseteq I$
and $I+I \subseteq I$. For a tuple
$p=(p_1,\ldots,p_k)$ of polynomials in $\re[x,y]$,
$\mbox{Ideal}(p)$ denotes the smallest ideal containing all $p_i$,
which is the set
\[
p_1 \cdot \re[x,y] + \cdots + p_k  \cdot \re[x,y].
\]
In computation, we often need to work with the {\it truncation}:
\[
\mbox{Ideal}(p)_{2k}  \, := \,
p_1 \cdot \re[x,y]_{2k-\deg(p_1)} + \cdots + p_k  \cdot \re[x,y]_{2k-\deg(p_k)}.
\]
For an ideal $I \subseteq \re[x,y]$,
its complex and real varieties are defined respectively as
\begin{align*}
\mc{V}_{\cpx}(I) := \{(u,v)\in \cpx^n \times \cpx^m \mid \,
f(u,v) = 0 \, \forall \, f \in I \}, \\
\mc{V}_{\re}(I)  := \{(u,v)\in \re^n \times \re^m \mid \,
f(u,v) = 0 \, \forall \, f \in I \}.
\end{align*}

A polynomial $\sig$ is said to be a sum of squares (SOS)
if $\sig = s_1^2+\cdots+ s_k^2$ for some real polynomials $s_1,\ldots, s_k$.
Whether or not a polynomial is SOS can be checked
by solving a semidefinite program (SDP) \cite{Las01,ParMP}.
Clearly, if a polynomial is SOS, then it is nonnegative everywhere.
However, the reverse may not be true.
Indeed, there are significantly more nonnegative polynomials than
SOS ones \cite{Blek06,BPT13}.
The set of all SOS polynomials in $(x,y)$ is denoted as $\Sig[x,y]$,
and its $d$th truncation is
\[
\Sig[x,y]_d := \Sig[x,y] \cap \re[x,y]_d.
\]
For a tuple $q=(q_1,\ldots,q_t)$ of polynomials in $(x,y)$,
its {\it quadratic module} is
\[
\mbox{Qmod}(q):=  \Sig[x,y] + q_1 \cdot \Sig[x,y] + \cdots + q_t \cdot \Sig[x,y].
\]
We often need to work with the truncation
\[
\mbox{Qmod}(q)_{2k} \, := \,
\Sig[x,y]_{2k} + q_1 \cdot \Sig[x,y]_{2k - \deg(g_1)}
+ \cdots + q_t \cdot \Sig[x,y]_{2k - \deg(q_t)}.
\]
For two tuples $p=(p_1,\ldots,p_k)$ and $q=(q_1,\ldots,q_t)$
of polynomials in $(x,y)$, for convenience, we denote
\be \label{df:IQ(h,g)}
\left\{ \baray{lcl}
\mbox{IQ}(p,q) &:=& \mbox{Ideal}(p) + \mbox{Qmod}(q), \\
\mbox{IQ}(p,q)_{2k} &:=& \mbox{Ideal}(p)_{2k} + \mbox{Qmod}(q)_{2k}.
\earay \right.
\ee
The set $\mbox{IQ}(p,q)$ (resp., $\mbox{IQ}(p,q)_{2k}$)
is a convex cone that is contained in $\re[x,y]$ (resp., $\re[x,y]_{2k}$).

The set $\mbox{IQ}(p,q)$ is said to be {\it archimedean}
if there exists $\sig \in \mbox{IQ}(p,q)$ such that
$\sig(x,y) \geq 0$ defines a compact set in $\re^n \times \re^m$.
If $\mbox{IQ}(p,q)$ is archimedean, then the set
$K: = \{p(x,y)=0, \, q(x,y)\geq 0 \}$ must be compact.
The reverse is not always true. However, if $K$ is compact, say,
$K \subseteq B(0,R)$
(the ball centered at $0$ with radius $R$), then
$\mbox{IQ}(p, \tilde{q})$ is always archimedean,
with $\tilde{q} = (q, R-\|x\|^2 - \|y\|^2)$,
while $\{p(x,y)=0, \, \tilde{q}(x,y) \geq 0 \}$
defines the same set $K$. Under the assumption that
$\mbox{IQ}(p,q)$ is archimedean,
every polynomial in $(x,y)$,
which is strictly positive on $K$,
must belong to $\mbox{IQ}(p,q)$.
This is the so-called Putinar's Positivstellensatz~\cite{Put}.
Interestingly, under some optimality conditions,
if a polynomial is nonnegative (but not strictly positive) over $K$,
then it belongs to $\mbox{IQ}(p,q)$.
This is shown in \cite{Nie-opcd}.

The above is for polynomials in $(x,y)$.
For polynomials in only $x$ or $y$,
the ideals, sum-of-squares, quadratic modules,
and their truncations are defined in the same way.
The notation $\Sig[x],\Sig[x]_d, \Sig[y],\Sig[y]_d$ are similarly defined.

\subsection{Localizing and moment matrices}

Let $\xi :=(\xi_1, \ldots, \xi_l)$ be a subvector of
$(x,y):=(x_1,\ldots,x_n, y_1, \ldots, y_m)$.
Throughout the paper, the vector $\xi$
is either $x$, or $y$, or $(x,y)$.
Denote by $\re^{\N_d^l}$ the space of real sequences
indexed by $\af \in \N_d^l$.
A vector in $w :=(w_\af)_{ \af \in \N_d^l} \in \re^{\N_d^l}$ is called a
{\it truncated multi-sequence} (tms) of degree $d$.
It gives the {\it Riesz functional}
$\mathscr{R}_w$ acting on $\re[\xi]_d$ as (each $f_\af \in \re$)
\be
\mathscr{R}_w \Big(\sum_{\af \in \N_d^l}
f_\af \xi^\af  \Big)\, := \, \sum_{\af \in \N_d^l}  f_\af w_\af.
\ee
For $f \in \re[\xi]_d$ and $w \in \re^{\N_d^l}$, we denote
\be \label{df:<p,y>}
\langle f, w \rangle \,:=\, \mathscr{R}_w(f).
\ee
Consider a polynomial $q \in \re[\xi]_{2k}$
with $\deg(q) \leq 2k$.
The $k$th {\it localizing matrix} of $q$,
generated by a tms $w \in \re^{\N^l_{2k}}$,
is the symmetric matrix $L_q^{(k)}(w)$ such that
\be  \label{LocMat}
vec(a_1)^T \Big( L_q^{(k)}(w) \Big) vec(a_2)  = \mathscr{R}_w(q a_1 a_2)
\ee
for all $a_1,a_2 \in \re[\xi]_{k - \lceil \deg(q)/2 \rceil}$.
(The $vec(a_i)$ denotes the coefficient vector of $a_i$.)
For instance, when $n=2$ and $k=2$ and $q = 1 - x_1^2-x_2^2$, we have
\[
L_q^{(2)}[w]=\left [
\begin{matrix}
w_{00}-w_{20}-w_{02} &  w_{10}-w_{30}-w_{12} &  w_{01}-w_{21}-w_{03} \\
w_{10}-w_{30}-w_{12} &  w_{20}-w_{40}-w_{22} &  w_{11}-w_{31}-w_{13} \\
w_{01}-w_{21}-w_{03} &  w_{11}-w_{31}-w_{13} &  w_{02}-w_{22}-w_{04} \\
\end{matrix}\right ].
\]
When $q = 1$ (the constant one polynomial),
$L_q^{(k)}(w)$ is called the {\it moment matrix} and we denote
\be \label{MomMat}
M_k(w):= L_{1}^{(k)}(w).
\ee
The columns and rows of $L_q^{(k)}(w)$, as well as $M_k(w)$,
are labeled by $\af \in \N^l$ with $2|\af|  \leq 2k - \deg(q) $.
When $q = (q_1, \ldots, q_t)$ is a tuple of polynomials,
then we define
\be  \label{block:LM}
L_q^{(k)}(w) \, := \, \mbox{diag}\Big(
L_{q_1}^{(k)}(w), \ldots,  L_{q_t}^{(k)}(w)
\Big),
\ee
which is a block diagonal matrix.
Moment and localizing matrices are important tools for constructing
semidefinite programming relaxations for
solving moment and polynomial optimization problems
\cite{FiaNie12,HelNie12,Las01,linopt15}.
Moreover, moment matrices are also useful
for computing tensor decompositions \cite{NieGP17}.
We refer to \cite{Todd}
for a survey on semidefinite programming and applications.

\section{An algorithm for solving SPPPs}
\label{sc:algSPPP}

Let $F,g,h$ be the polynomial tuples for the saddle point problem
\reff{dfsad:F:XY}. Assume $g,h$ are nonsingular.
So the Lagrange multiplier vectors $\lmd(x,y),\mu(x,y)$ can be expressed as in
\reff{lmd(x,y)}-\reff{mu(x,y)}. We have seen that
each saddle point $(x^*,y^*)$ must satisfy \reff{SadKKT:XY}.
This leads us to consider the optimization problem
\be  \label{min:F(xy):kt}
\left\{ \baray{cl}
\min\limits_{x\in X, y\in Y}  &  F(x,y) \\
\mbox{subject to} & \nabla_x F(x,y) - \sum_{i \in \mc{E}^X} \lmd_i(x,y) \nabla_x g_i(x) = 0, \\
     &  \nabla_y F(x,y) - \sum_{j \in \mc{E}^Y} \mu_j(x,y) \nabla_y h_j(y) = 0, \\
     &  0 \leq \lmd_i(x,y) \perp  g_i(x)  \geq 0 \, (i \in \mc{E}_2^X ), \\
     &  0 \geq \mu_j(x,y) \perp  h_j(y)  \geq 0 ( j \in \mc{E}_2^Y),  \\
\earay \right.
\ee
where $\lmd_i(x,y)$ and $\mu_j(x,y)$ are Lagrange polynomials
given as in \reff{lmd(x,y)}-\reff{mu(x,y)}.
The saddle point problem \reff{dfsad:F:XY}
is not equivalent to \reff{min:F(xy):kt}.
However, the optimization problem~\reff{min:F(xy):kt}
can be used to get a candidate saddle point.
Suppose $(x^*,y^*)$ is a minimizer of \reff{min:F(xy):kt}.
If $x^*$ is a minimizer of $F(x,y^*)$ over $X$
and $y^*$ is a maximizer of $F(x^*,y)$ over $Y$,
then $(x^*,y^*)$ is a saddle point;
otherwise, such $(x^*,y^*)$ is not a saddle point,
i.e., there exists $u \in X$ and/or there exists $v \in Y$ such that
\[
F(u,y^*) - F(x^*, y^*) < 0 \quad \mbox{and/or} \quad
F(x^*,v) - F(x^*, y^*) > 0.
\]
The points $u,v$ can be used to give new constraints
\be \label{new:valid:ineq}
F(u,y) - F(x, y) \ge 0 \quad \mbox{and/or} \quad
F(x,y) - F(x, v) \ge 0.
\ee
Every saddle point $(x,y)$ must satisfy \reff{new:valid:ineq},
so \reff{new:valid:ineq} can be added to the optimization problem~\reff{min:F(xy):kt}
without excluding any true saddle points.
For generic polynomials $F,g,h$, the problem \reff{min:F(xy):kt}
has only finitely many feasible points (see Theorem~\ref{thm:Cktpt:fnt}).
Therefore, by repeatedly adding new inequalities like \reff{new:valid:ineq},
we can eventually get a saddle point or detect nonexistence of saddle points.
This results in the following algorithm.

\begin{alg}  \label{alg:pop:kt}
(An algorithm for solving saddle point problems.)

\bit

\item[\textbf{Input:}] The polynomials $F, g, h$
as in \reff{dfsad:F:XY}, \reff{set:X}, \reff{set:Y}
and Lagrange multiplier expressions as in
\reff{lmd(x,y)}-\reff{mu(x,y)}.

\item[Step~0:] Let $K_1=K_2 = \mc{S}_a := \emptyset$ be empty sets.

\item [Step~1:] If the problem~\reff{min:F(xy):kt} is infeasible,
then \reff{dfsad:F:XY} does not have a saddle point and stop;
otherwise, solve \reff{min:F(xy):kt}
for a set $K^0$ of minimizers. Let $k:=0$.

\item [Step~2:] For each $(x^*, y^*) \in K^k$, do the following:
\bit
\item [(a):] (Lower level minimization)
Solve the problem
\be \label{min:F(x,y*)}
\left\{ \baray{rl}
\vartheta_1(y^*):= \min\limits_{x \in X}   &  F(x,y^*) \\
\mbox{subject to} & \nabla_x F(x,y^*) -
\sum_{i \in \mc{E}^X} \lmd_i(x,y^*) \nabla_x g_i(x) = 0, \\
     &  0 \leq \lmd_i(x,y^*) \perp  g_i(x)  \geq 0 \, (i \in \mc{E}_2^X ),\\
\earay \right.
\ee
and get a set of minimizers $S_1(y^*)$.
If $F(x^*,y^*) > \vartheta_1(y^*)$, update
\[
K_1  \, := \, K_1 \cup S_1(y^*).
\]

\item [(b):] (Lower level maximization)
Solve the problem
\be \label{max:F(x*,y)}
\left\{ \baray{rl}
\vartheta_2(x^*):= \max\limits_{y \in Y}  &  F(x^*,y) \\
\mbox{subject to} & \nabla_y F(x^*,y) - \sum_{j \in \mc{E}^Y}
\mu_j(x^*,y) \nabla_y h_j(y) = 0, \\
&  0 \geq \mu_j(x^*,y) \perp  h_j(y)  \geq 0 ( j \in \mc{E}_2^Y) \\
\earay \right.
\ee
and get a set of maximizers $S_2(x^*)$.
If $F(x^*,y^*) < \vartheta_2(x^*)$, update
\[
K_2  \, := \, K_2 \cup S_2(x^*).
\]

\item [(c):]
If $\vartheta_1(y^*) = F(x^*,y^*) = \vartheta_2(x^*)$, update:
\[
\mc{S}_a := \mc{S}_a \cup \{ (x^*, y^*) \}.
\]

\eit

\item [Step~3:] If $\mc{S}_a \ne \emptyset$,
then each point in $\mc{S}_a$ is a saddle point and stop;
otherwise go to Step~4.

\item [Step~4:] (Upper level minimization)
Solve the optimization problem
\be  \label{minF:KT:uv}
\left\{ \baray{cl}
\min\limits_{x\in X, y\in Y} &  F(x,y) \\
\mbox{subject to} & \nabla_x F(x,y) - \sum_{i \in \mc{E}^X} \lmd_i(x,y) \nabla_x g_i(x) = 0, \\
     &  \nabla_y F(x,y) - \sum_{j \in \mc{E}^Y} \mu_j(x,y) \nabla_y h_j(y) = 0, \\
     &  0 \leq \lmd_i(x,y) \perp  g_i(x)  \geq 0 \, (i \in \mc{E}_2^X ), \\
     &  0 \geq \mu_j(x,y) \perp  h_j(y)  \geq 0 ( j \in \mc{E}_2^Y),   \\
     &  F(u,y) - F(x,y) \geq 0 \, ( u \in K_1),  \\
     &  F(x,v) - F(x,y) \leq 0 \, ( v \in K_2).  \\
\earay \right.
\ee
If \reff{minF:KT:uv} is infeasible,
then \reff{dfsad:F:XY} has no saddle points and stop;
otherwise, compute a set $K^{k+1}$
of optimizers for \reff{minF:KT:uv}.
Let $k:=k+1$ and go to Step~2.

\item[\textbf{Output:}]
If $\mc{S}_a$ is nonempty,
every point in $\mc{S}_a$ is a saddle point;
otherwise, output that there is no saddle point.

\eit
\end{alg}

For generic polynomials, the feasible set $\mc{K}_0$
of \reff{min:F(xy):kt}, as well as
each $K^k$ in Algorithm~\ref{alg:pop:kt}, is finite.
The convergence of Algorithm~\ref{alg:pop:kt} is shown as follows.

\begin{theorem} \label{thm:fnt:iter}
Let $\mc{K}_0$ be the feasible set of \reff{min:F(xy):kt}
and let $\mc{S}_a$ be the set of saddle points for \reff{dfsad:F:XY}.
If the complement set of $\mc{S}_a$ in $\mc{K}_0$
(i.e., the set $\mc{K}_0 \setminus \mc{S}_a$) is finite,
then Algorithm~\ref{alg:pop:kt} must terminate
after finitely many iterations.
Moreover, if $\mc{S}_a \ne \emptyset$,
then each $(x^*, y^*) \in \mc{S}_a$ is a saddle point;
if $\mc{S}_a = \emptyset$, then there is no saddle point.
\end{theorem}
\begin{proof}
At an iteration, if $\mc{S}_a \ne \emptyset$,
then Algorithm~\ref{alg:pop:kt} terminates.
For each iteration with $\mc{S}_a = \emptyset$,
each point $(x^*, y^*) \in K^k$ is not feasible for \reff{minF:KT:uv}.
When the $k$th iteration goes to the $(k+1)$th one, the nonempty sets
\[
K^0, \quad K^1, \quad K^2, \quad K^3, \ldots, \quad K^k
\]
are disjoint from each other. All the points in $K^i$
are not saddle points, so
\[
\bigcup_{i=0}^k K^i \subseteq \mc{K}_0 \setminus \mc{S}_a.
\]
Therefore, when the set $\mc{K}_0 \setminus \mc{S}_a$ is finite,
Algorithm~\ref{alg:pop:kt} must terminate after finitely many iterations.

When $\mc{S}_a \ne \emptyset$, each point $(x^*, y^*) \in \mc{S}_a$
is verified as a saddle point in Step~2.
When $\mc{S}_a = \emptyset$, Algorithm~\ref{alg:pop:kt}
stops in Step~4 at some iteration,
with the case that \reff{minF:KT:uv} is infeasible.
Since every saddle point is feasible for both
\reff{min:F(xy):kt} and \reff{minF:KT:uv},
there does not exist a saddle point if $\mc{S}_a = \emptyset$.
\end{proof}

The number of iterations required by Algorithm~\ref{alg:pop:kt}
to terminate is bounded above by the cardinality of the complement set
$\mc{K}_0 \setminus \mc{S}_a$, which is always less than or equal to
the cardinality $|\mc{K}_0|$ of the feasible set of \reff{min:F(xy):kt}.
Generally, it is hard to count
$|\mc{K}_0 \setminus \mc{S}_a|$ or $|\mc{K}_0|$ accurately.
When the polynomials $F,g,h$ are generic,
we can prove that the number of solutions
for equality constraints in \reff{min:F(xy):kt} is finite.
For degrees $a_0, b_0 >0$, denote the set product
$\cpx[x,y]_{a_0, b_0} := \cpx[x]_{a_0} \cdot \cpx[y]_{b_0}$.

\begin{theorem} \label{thm:Cktpt:fnt}
Let $a_i, b_j >0$ be positive degrees, for
$i \in \mc{E}^X$ and $j \in \mc{E}^Y$.
If $F(x,y) \in \cpx[x,y]_{a_0,b_0}$,
$g_i \in \cpx[x]_{a_i}$, $h_j\in \cpx[y]_{b_j}$
are generic polynomials, then the polynomial system
\be \label{cpxkkt:Fxy}
\left\{ \baray{l}
\nabla_xF(x, y) = \sum_{ i \in \mc{E}^X }  \lmd_i(x,y) \nabla_x g_i(x) , \\
 g_i(x)  = 0 \, (i \in \mc{E}^X_1), \,
 \lmd_i(x,y)   g_i(x)  = 0 \, (i \in \mc{E}^X_2), \\
\nabla_yF(x, y) = \sum_{j \in \mc{E}^Y } \mu_j(x,y)  \nabla_y h_j(y), \\
 h_j(y)  = 0 \, (j \in \mc{E}^Y_1), \,
 \mu_j(x,y) h_j(y)  = 0 \, (j \in \mc{E}^Y_2)
\earay\right.
\ee
has only finitely many complex solutions in $\cpx^n \times \cpx^m$.
\end{theorem}

The proof for Theorem~\ref{thm:Cktpt:fnt} will be given in Section~\ref{sc:proof}.
One would like to know
what is the number of complex solutions
to the polynomial system~\reff{cpxkkt:Fxy} for generic polynomials $F,g,h$.
That number is an upper bound for $|\mc{K}_0|$
and so is also an upper bound for the number of iterations required
by Algorithm~\ref{alg:pop:kt} to terminate.
The following theorem gives an upper bound for $|\mc{K}_0|$.

\begin{theorem}  \label{thm:iter:bound}
For the degrees $a_i, b_j$ as in Theorem~\ref{thm:Cktpt:fnt},  let
\be \label{upbd:M}
M \, := \,
\sum_{ \substack{
\{i_1, \ldots, i_{r_1} \} \subseteq [\ell_1], 0\leq r_1 \leq n \\
\{j_1, \ldots, j_{r_2} \} \subseteq [\ell_2], 0\leq r_2 \leq m
}}
a_{i_1} \cdots a_{ i_{r_1} } b_{j_1} \cdots b_{ j_{r_2} }
\cdot s
\ee
where in the above the number $s$ is given as
\[
s = \sum_{ \substack{
k_0 + \cdots + k_{r_1+r_2} = n+m-r_1-r_2 \\
k_0, \ldots, k_{r_1+r_2} \in \N
}}
(a_0+b_0)^{k_0}
(a_{i_1})^{k_1} \cdots (a_{ i_{r_1} })^{k_{r_1}}
(b_{j_1})^{ k_{r_1+1} } \cdots (b_{ j_{r_2} })^{ k_{r_1+r_2} } .
\]
If $F(x,y)$, $g_i$, $h_j$ are generic, then \reff{cpxkkt:Fxy}
has at most $M$ complex solutions,
and hence Algorithm~\ref{alg:pop:kt} must terminate within $M$ iterations.
\end{theorem}

The proof for Theorem~\ref{thm:iter:bound}
will be given in Section~\ref{sc:proof}.
We remark that the upper bound
$M$ given in \reff{upbd:M} is not sharp.
In our computational practice, Algorithm~\ref{alg:pop:kt}
typically terminates after a few iterations.
It is an interesting question to obtain accurate upper bounds
for the number of iterations required by Algorithm~\ref{alg:pop:kt} to terminate.

\section{Solving optimization problems}
\label{sc:slvopt}

We discuss how to solve the optimization problems
that appear in Algorithm~\ref{alg:pop:kt}.
Under some genericity assumptions on $F,g,h$,
we show that their optimizers can be computed
by solving Lasserre type semidefinite relaxations.
Let $X,Y$ be feasible sets given as in \reff{set:X}-\reff{set:Y}.
Assume $g,h$ are nonsingular, so $\lmd(x,y), \mu(x,y)$
can be expressed as in \reff{lmd(x,y)}-\reff{mu(x,y)}.

\subsection{The upper level optimization}

The optimization problem~\reff{min:F(xy):kt}
is a special case of \reff{minF:KT:uv},
with $K_1=K_2 = \emptyset$.
It suffices to discuss how to solve \reff{minF:KT:uv}
with finite sets $K_1,K_2$.
For convenience, we rewrite \reff{minF:KT:uv} explicitly as
\be  \label{optF:KKT:xy}
\left\{ \baray{cl}
\min\limits_{(x,y)} &  F(x,y) \\
\mbox{subject to} & \nabla_x F(x,y) - \sum_{i \in \mc{E}^X} \lmd_i(x,y) \nabla_x g_i(x) = 0, \\
     & \nabla_y F(x,y) - \sum_{j \in \mc{E}^Y} \mu_j(x,y) \nabla_y h_j(y) = 0, \\
     &  g_i(x)  = 0 ,\,  h_j(y) = 0 \, ( i \in \mc{E}_1^X,  j \in \mc{E}_1^Y), \\
     & \lmd_i(x,y)  g_i(x)  = 0, \,
\mu_j(x,y) h_j(y)  = 0 \, (i \in \mc{E}_2^X, j \in \mc{E}_2^Y),   \\
     & g_i(x) \geq 0,\, \lmd_i(x,y) \geq 0 \, (i \in \mc{E}_2^X ), \\
     & h_j(y) \geq 0,\, -\mu_j(x,y)  \geq 0 \, ( j \in \mc{E}_2^Y),  \\
     &  F(u,y) - F(x,y) \geq 0 \, (\forall \, u \in K_1),  \\
     &  F(x,y) - F(x,v) \geq 0 \, (\forall \, v \in K_2).  \\
\earay \right.
\ee
Recall that $\lmd_i(x,y)$, $\mu_j(x,y)$ are Lagrange polynomials
as in \reff{lmd(x,y)}-\reff{mu(x,y)}.
Denote by $\phi$ the tuple of equality constraining polynomials
\begin{multline}  \label{phi:xy}
\phi := \Big\{
     \nabla_x F -{\sum}_{i \in \mc{E}^X} \lmd_i(x,y) \nabla_x g_i \Big\}
\cup \Big\{ \nabla_y F - {\sum}_{j \in \mc{E}^Y} \mu_j(x,y) \nabla_y h_j \Big\} \\
\cup \Big\{ g_i, h_j \Big\}_{ i \in \mc{E}_1^X, j \in \mc{E}_1^Y }
\cup \Big\{ \lmd_i(x,y)  g_i,\,
  \mu_j(x,y) h_j \Big\}_{i \in \mc{E}_2^X, j \in \mc{E}_2^Y  } ,
\end{multline}
and denote by $\psi$ the tuple of inequality constraining ones
\begin{multline}  \label{psi:xy}
\psi := \Big\{
g_i, \, h_j, \, \lmd_i(x,y), \, -\mu_j(x,y)
\Big\}_{i \in \mc{E}_2^X, j \in \mc{E}_2^Y } \cup \\
\Big\{ F(u,y)-F(x,y), \, F(x,y)-F(x,v) \Big\}_{u \in K_1, \, v \in K_2}.
\end{multline}
They are polynomials in $(x,y)$. Let
\be \label{d0:phi:psi}
d_0 \, := \, \big \lceil \frac{1}{2}\max\{\deg{F(x,y)},
 \deg(\phi), \deg(\psi)\} \big \rceil.
\ee
Then, the optimization problem~\reff{optF:KKT:xy} can be simply written as
\be \label{minF(xy):phi:psi}
\left\{ \baray{rl}
f_{\ast}:= \min & F(x,y) \\
\mbox{subject to} &  \phi(x,y) = 0, \, \psi(x,y) \geq 0.
\earay \right.
\ee
We apply Lasserre's hierarchy of semidefinite relaxations
to solve \reff{minF(xy):phi:psi}.
For integers $k=d_0, d_0+1,\cdots$,
the $k$th order semidefinite relaxation is
\be  \label{<F,w>:lasrlx:ordk}
\left\{ \baray{rl}
F_k:= \min & \langle F, w \rangle \\
\mbox{subject to} &   (w)_0 = 1, \, M_k(w) \succeq 0, \\
 &  L^{(k)}_{\phi} (w) = 0, L^{(k)}_{\psi} (w) \succeq 0, \,
   w \in \re^{ \N_{2k}^{n+m} }.
\earay \right.
\ee
The number $k$ is called a relaxation order.
We refer to \reff{LocMat} for the
localizing and moment matrices used in \reff{<F,w>:lasrlx:ordk}.
\begin{alg} \label{alg:LaSDP:xy}
(An algorithm for solving the optimization \reff{optF:KKT:xy}.)
\bit

\item[\textbf{Input:}] Polynomials $F, \phi, \psi$ as in \reff{phi:xy}-\reff{psi:xy}.

\item [Step~0:] Let $k:=d_0$.

\item [Step~1:] Solve the semidefinite relaxation \reff{<F,w>:lasrlx:ordk}.

\item [Step~2:] If the relaxation \reff{<F,w>:lasrlx:ordk} is infeasible,
then \reff{dfsad:F:XY} has no saddle points and stop;
otherwise, solve it for a minimizer $w^*$. Let $t := d_0$.

\item [Step~3] Check whether or not $w^*$ satisfies the rank condition
\be \label{cond:flat}
\rank \, M_{t} (w^*)  \, = \, \rank \, M_{t-d_0} (w^*).
\ee

\item [Step~4]
If \reff{cond:flat} holds, extract $r := \rank \, M_{t} (w^*)$
minimizers for \reff{optF:KKT:xy} and stop.

\item [Step~5] If $t< k$, let $t:=t+1$ and go to Step~3;
otherwise, let $k:=k+1$ and go to Step~1.

\item[\textbf{Output:}]
Minimizers of the optimization problem \reff{optF:KKT:xy}
or a certificate for the infeasibility of \reff{optF:KKT:xy}.

\eit

\end{alg}

The conclusions in the Steps~2 and 3
are justified by the following Proposition~\ref{prop:Alg4.1}.
The rank condition~\reff{cond:flat} is called
flat extension or flat truncation \cite{CuFi05,Nie-ft}.
It is a sufficient and also almost necessary criterion
for checking convergence of Lasserre type relaxations \cite{Nie-ft}.
When it is satisfied, the method in \cite{HenLas05}
can be applied to extract minimizers in Step~4.
It was implemented in the software {\tt GloptiPoly~3} \cite{GloPol3}.

\begin{prop} \label{prop:Alg4.1}
Suppose $g,h$ are nonsingular polynomial tuples.
For the hierarchy of relaxations~\reff{<F,w>:lasrlx:ordk},
we have the properties:

\bit

\item [i)] If \reff{<F,w>:lasrlx:ordk} is infeasible for some $k$,
then \reff{optF:KKT:xy} is infeasible
and \reff{dfsad:F:XY} has no saddle points.

\item [ii)] If \reff{<F,w>:lasrlx:ordk} has a minimizer $w^*$
satisfying \reff{cond:flat},
then $F_k = f_{\ast}$ and there are
$r := \rank \, M_{t} (w^*)$ minimizers for \reff{optF:KKT:xy}.

\eit

\end{prop}

\begin{proof}
Since $g,h$ are nonsingular,
every saddle point must be a critical point,
and Lagrange multipliers can be expressed as in
\reff{lmd(x,y)}-\reff{mu(x,y)}.

i) For each $(u,v)$ that is feasible for \reff{optF:KKT:xy},
$[(u,v)]_{2k}$ satisfies all the constraints of
\reff{<F,w>:lasrlx:ordk}, for all $k$.
Therefore, if \reff{<F,w>:lasrlx:ordk} is infeasible for some $k$,
then \reff{optF:KKT:xy} is infeasible.

ii) The conclusion follows from the classical results in
\cite{CuFi05,HenLas05,Lau05,Nie-ft}.
\end{proof}

We refer to \reff{df:IQ(h,g)} for the notation $\mbox{IQ}$,
which is the sum of an ideal and a quadratic module.
The polynomial tuples $\phi,\psi$ are from \reff{phi:xy}-\reff{psi:xy}.
Algorithm~\ref{alg:LaSDP:xy} is able to solve \reff{optF:KKT:xy}
successfully after finitely many iterations,
under the following genericity conditions.

\begin{cond} \label{cond:AC:FXYuv}
The polynomial tuples $g,h$ are nonsingular and $F,g,h$
satisfy one (not necessarily all) of the following:
\bit

\item [(1)] $\mbox{IQ}(g_{eq},g_{in})+\mbox{IQ}(h_{eq}, h_{in})$ is archimedean;

\item [(2)] the equation $\phi(x,y)=0$ has
finitely many real solutions;

\item [(3)] $\mbox{IQ}(\phi,\psi)$ is archimedean.

\eit

\end{cond}

In the above, the item (1) is almost the same as that
$X,Y$ are compact sets; the item (2) is the same as that
\reff{cpxkkt:Fxy} has only finitely many real solutions.
Also note that the item (1) or (2) implies (3).
In Theorem~\ref{thm:Cktpt:fnt}, we have shown that \reff{cpxkkt:Fxy}
has only finitely many complex solutions when $F,g,h$ are generic.
Therefore, Condition~\ref{cond:AC:FXYuv} holds generically.
Under Condition~\ref{cond:AC:FXYuv}, Algorithm~\ref{alg:LaSDP:xy}
can be shown to have finite convergence.

\begin{theorem} \label{thm:Lascvg:xy}
Under Condition~\ref{cond:AC:FXYuv}, we have that:
\bit

\item [i)] If the problem~\reff{optF:KKT:xy} is infeasible,
then the semidefinite relaxation \reff{<F,w>:lasrlx:ordk}
must be infeasible for all $k$ big enough.

\item [ii)]
Suppose \reff{optF:KKT:xy} is feasible.
If \reff{optF:KKT:xy} has only finitely many minimizers
and each of them is an isolated critical point
(i.e., an isolated real solution of \reff{cpxkkt:Fxy}),
then, for all $k$ big enough, \reff{<F,w>:lasrlx:ordk}
has a minimizer and each minimizer must satisfy
the rank condition \reff{cond:flat}.

\eit
\end{theorem}

We would like to remark that when $F,g,h$ are generic,
every minimizer of \reff{optF:KKT:xy}
is an isolated real solution of \reff{cpxkkt:Fxy}.
This is because \reff{cpxkkt:Fxy} has only finitely many complex solutions
for generic $F,g,h$. Therefore,
Algorithm~\ref{alg:LaSDP:xy} has finite convergence for generic cases.
We would also like to remark that Proposition~\ref{prop:Alg4.1}
and Theorem~\ref{thm:Lascvg:xy} assume that
the semidefinite relaxation~\reff{<F,w>:lasrlx:ordk}
is solved exactly. However,
semidefinite programs are usually solved numerical (e.g., by {\tt SeDuMi}),
for better computational performance.
Therefore, in computational practice,
the optimizers obtained by Algorithm~\ref{alg:LaSDP:xy}
are correct up to numerical errors.
This is a common feature of all numerical methods.

\subsection{Lower level minimization}
\label{ssc:lowmin}

For a given pair $(x^*,y^*)$
that is feasible for \reff{min:F(xy):kt} or \reff{minF:KT:uv},
we need to check whether or not $x^*$ is a minimizer of $F(x,y^*)$
over $X$. This requires us to solve the minimization problem
\be \label{minF(xy*):X}
\left\{ \baray{cl}
\min\limits_{x \in \re^n} & F(x,y^*) \\
\mbox{subject to} &   g_i(x)  = 0 \, (i \in \mc{E}_1^X ), \\
& g_i(x)  \geq 0 \, (i \in \mc{E}_2^X ).
\earay \right.
\ee
When $g$ is nonsingular, if it has a minimizer,
the optimization \reff{minF(xy*):X} is equivalent to
(by adding necessary optimality conditions)
\be \label{F(x,y*):min}
\left\{ \baray{cl}
\min\limits_{x \in \re^n}   &  F(x,y^*) \\
\mbox{subject to} & \nabla_x F(x,y^*) - \sum\limits_{i \in \mc{E}^X}
                         \lmd_i(x,y^*) \nabla_x g_i(x) = 0, \\
     &  g_i(x)  = 0 \, (i \in \mc{E}_1^X ), \,
     \lmd_i(x,y^*)  g_i(x)  = 0 \, (i \in \mc{E}_2^X ),\\
     &  g_i(x)  \geq 0,\, \lmd_i(x,y^*) \geq 0 \, (i \in \mc{E}_2^X ). \\
\earay \right.
\ee
Denote the tuple of equality constraining polynomials
\begin{multline} \label{phi(x):y*}
\phi_{y^*} \, := \, \Big\{
\nabla_x F(x,y^*) - {\sum}_{i \in \mc{E}^X}
\lmd_i(x,y^*) \nabla_x g_i \Big\}  \\
\cup \big\{g_i \big\}_{i \in \mc{E}_1^X } \cup \big\{
\lmd_i(x,y^*) \cdot g_i \big\}_{i \in \mc{E}_2^X },
\end{multline}
and denote the tuple of inequality ones
\be \label{psi(x):y*}
\psi_{y^*} \, := \, \Big\{
g_i,\, \lmd_i(x,y^*) \Big\}_{i \in \mc{E}_2^X }.
\ee
They are polynomials in $x$ but not in $y$,
depending on the value of $y^*$. Let
\be \label{d1:phi:psi}
d_1 \, := \, \big \lceil \frac{1}{2}\max\{\deg{F(x,y^*)},
          \deg(\phi_{y^*}), \deg(\psi_{y^*})\} \big \rceil.
\ee
We can rewrite \reff{F(x,y*):min} equivalently as
\be \label{minF(x):phipsi:y*}
\left\{ \baray{cl}
\min\limits_{x \in \re^n}  &  F(x,y^*)  \\
 \mbox{subject to} &
\phi_{y^*}(x) = 0, \, \psi_{y^*}(x) \geq 0.
\earay \right.
\ee
Lasserre's hierarchy of semidefinite relaxations
for solving \reff{minF(x):phipsi:y*} is
\be  \label{min<F(xy*),z>}
\left\{ \baray{cl}
 \min\limits_{z}  & \langle F(x,y^*), z \rangle \\
\mbox{subject to} &   (z)_0 = 1, \, M_k(z) \succeq 0, \\
 &  L^{(k)}_{\phi_{y^*} } (z) = 0,
 L^{(k)}_{\psi_{y^*} } (z) \succeq 0, \,  z \in \re^{ \N^n_{2k}  },
\earay \right.
\ee
for relaxation orders $k = d_1, d_1+1, \cdots$.
Since $(x^*,y^*)$ is a feasible pair for
\reff{min:F(xy):kt} or \reff{minF:KT:uv},
the problems \reff{minF(xy*):X} and \reff{minF(x):phipsi:y*}
are also feasible, hence \reff{min<F(xy*),z>} is also feasible.
A standard algorithm for solving
\reff{minF(x):phipsi:y*} is as follows.

\begin{alg}  \label{alg:min:F(xy*)}
(An algorithm for solving the problem \reff{minF(x):phipsi:y*}.)

\bit

\item[\textbf{Input:}] The point $y^*$ and polynomials $F(x,y^*), \phi_{y^*}, \psi_{y^*} $ as in \reff{phi(x):y*}-\reff{psi(x):y*}.

\item [Step~0:] Let $k \, := \, d_1$.
\item [Step~1:] Solve the semidefinite relaxation
\reff{min<F(xy*),z>} for a minimizer $z^*$. Let $t := d_1$.

\item [Step~2:] Check whether or not $z^*$ satisfies the rank condition
\be \label{flatcd:F(x):X}
\rank \, M_{t} (z^*)  \, = \, \rank \, M_{t-d_1} (z^*).
\ee

\item [Step~3:] If \reff{flatcd:F(x):X} holds, extract
$r :=\rank \, M_{t} (z^*)$ minimizers and stop.

\item [Step~4:] If $t< k$, let $t:=t+1$ and go to Step~3;
otherwise, let $k:=k+1$ and go to Step~1.

\item[\textbf{Output:}] Minimizers of the optimization problem \reff{minF(x):phipsi:y*}.

\eit
\end{alg}

Similar conclusions as in Proposition~\ref{prop:Alg4.1}
hold for Algorithm~\ref{alg:min:F(xy*)}.
For cleanness of the paper, we do not state them again.
The method in \cite{HenLas05} can be applied to extract minimizers in the Step~3.
Moreover, Algorithm~\ref{alg:min:F(xy*)} also terminates within finitely
many iterations, under some genericity conditions.

\begin{cond} \label{cond:fix:y*}
The polynomial tuple $g$ is nonsingular and the point $y^*$
satisfies one (not necessarily all) of the following:

\bit

\item [(1)] $\mbox{IQ}(g_{eq}, g_{in})$ is archimedean;

\item [(2)] the equation $\phi_{y^*}(x)=0$ has
finitely many real solutions;

\item [(3)] $\mbox{IQ}(\phi_{y^*},\psi_{y^*})$ is archimedean.

\eit

\end{cond}

Since $(x^*, y^*)$ is feasible for \reff{min:F(xy):kt} or \reff{minF:KT:uv},
Condition~\ref{cond:AC:FXYuv} implies Condition~\ref{cond:fix:y*},
which also holds generically.
The finite convergence of Algorithm~\ref{alg:min:F(xy*)}
is summarized as follows.

\begin{theorem}
\label{thm:min:F(xy*)}
Assume the optimization problem \reff{minF(xy*):X} has a minimizer
and Condition~\ref{cond:fix:y*} holds.
If each minimizer of \reff{minF(xy*):X} is an isolated critical point,
then, for all $k$ big enough, \reff{min<F(xy*),z>}
has a minimizer and each of them must satisfy \reff{flatcd:F(x):X}.
\end{theorem}

The proof of Theorem~\ref{thm:min:F(xy*)}
will be given in Section~\ref{sc:proof}.
We would like to remark that every minimizer of \reff{minF(x):phipsi:y*}
is an isolated critical point of \reff{minF(xy*):X},
when $F,g,h$ are generic. This is implied by Theorem~\ref{thm:Cktpt:fnt}.

\subsection{Lower level maximization}
\label{ssc:lowmax}

For a given pair $(x^*, y^*)$ that is feasible for
\reff{min:F(xy):kt} or \reff{minF:KT:uv},
we need to check whether or not
$y^*$ is a maximizer of $F(x^*,y)$ over $Y$.
This requires us to solve the maximization problem
\be \label{minF(x*y):y:Y}
\left\{ \baray{cl}
\max\limits_{y \in \re^m}  & F(x^*,y) \\
\mbox{subject to}  & h_j(y) = 0 \, (j \in  \mc{E}_1^Y), \,
h_j(y) \geq 0 \, (j \in  \mc{E}_2^Y).
\earay \right.
\ee
When $h$ is nonsingular, if it has a minimizer,
the optimization \reff{minF(x*y):y:Y} is equivalent to
(by adding necessary optimality conditions) the problem
\be \label{F(x*,y):max}
\left\{ \baray{cl}
\max\limits_{y \in \re^m}  &  F(x^*,y) \\
\mbox{subject to} & \nabla_y F(x^*,y) -
    \sum_{j \in \mc{E}^Y} \mu_j(x^*,y) \nabla_y h_j(y) = 0, \\
&   h_j(y) = 0 \, ( j \in \mc{E}_1^Y), \,
\mu_j(x^*,y) \cdot  h_j(y) = 0 \, ( j \in \mc{E}_2^Y),  \\
&  h_j(y)  \geq 0, \,  -\mu_j(x^*,y) \geq 0\, ( j \in \mc{E}_2^Y).
\earay \right.
\ee
Denote the tuple of equality constraining polynomials
\begin{multline}
\label{phi(y):x*}
\phi_{x^*} := \Big\{
\nabla_y F(x^*,y) - \sum_{j \in \mc{E}^Y}
         \mu_j(x^*,y) \nabla_y h_j  \Big\}   \\
\cup \big\{ h_j \big\}_{ j \in \mc{E}_1^Y }
\cup \big\{ \mu_j(x^*,y) h_j \big\}_{ j \in \mc{E}_2^Y },
\end{multline}
and denote the tuple of inequality ones
\be
\label{psi(y):x*}
\psi_{x^*} :=
\Big\{ h_j, \, -\mu_j(x^*,y) \Big\}_{ j \in \mc{E}_2^Y }.
\ee
They are polynomials in $y$ but not in $x$,
depending on the value of $x^*$. Let
\be \label{d2:phi:psi}
d_2 \, := \, \big \lceil \frac{1}{2}\max\{\deg{F(x^*,y)},
      \deg(\phi_{x^*}), \deg(\psi_{x^*})\} \big \rceil.
\ee
Hence, \reff{F(x*,y):max} can be simply expressed as
\be \label{minF(y):ppsi:x*}
\left\{ \baray{cl}
\max\limits_{y \in \re^m} & F(x^*,y) \\
\mbox{subject to} &
\phi_{x^*}(y) = 0,\,  \psi_{x^*}(y) \geq 0.
\earay \right.
\ee
Lasserre's hierarchy of semidefinite relaxations
for solving \reff{minF(y):ppsi:x*} is
\be  \label{max:F(x*y):las}
\left\{ \baray{cl}
\max\limits_{z} & \langle F(x^*, y), z \rangle \\
\mbox{subject to} &   (z)_0 = 1, \, M_k(z) \succeq 0, \\
 &  L^{(k)}_{\phi_{x^*} } (z) = 0, L^{(k)}_{\psi_{x^*} } (z) \succeq 0, \\
 &  z \in \re^{ \N^m_{2k} },
\earay \right.
\ee
for relaxation orders $k=d_2, d_2+1, \cdots$.
Since $(x^*, y^*)$ is feasible for
\reff{min:F(xy):kt} or \reff{minF:KT:uv},
the problems \reff{minF(x*y):y:Y} and \reff{minF(y):ppsi:x*}
must also be feasible. Hence, the relaxation
\reff{max:F(x*y):las} is always feasible.
Similarly, an algorithm for solving \reff{minF(y):ppsi:x*} is as follows.

\begin{alg}
\label{alg:maxF(y):Y}
(An algorithm for solving the problem \reff{minF(y):ppsi:x*}.)
\bit

\item [\textbf{Input:}] The point $x^*$ and polynomials
$F(x^*,y),\phi_{x^*},\psi_{x^*} $ as in \reff{phi(y):x*}-\reff{psi(y):x*}.

\item [Step~0:] Let $k:=d_2$.

\item [Step~1:] Solve the semidefinite relaxation
\reff{max:F(x*y):las} for a maximizer $z^*$.
Let $t := d_2$.

\item [Step~2:] Check whether or not $z^*$ satisfies the rank condition
\be  \label{flat:F(x*y):Y}
\rank \, M_{t} (z^*)  \, = \, \rank \, M_{t-d_2} (z^*).
\ee

\item [Step~3:] If \reff{flat:F(x*y):Y} holds, extract
$r :=\rank \, M_{t} (z^*)$ maximizers for \reff{minF(y):ppsi:x*} and stop.

\item [Step~4:] If $t< k$, let $t:=t+1$ and go to Step~3;
otherwise, let $k:=k+1$ and go to Step~1.

\item [\textbf{Output:}]
Maximizers of the optimization problem \reff{minF(y):ppsi:x*}.

\eit
\end{alg}

The same kind of conclusions like in Proposition~\ref{prop:Alg4.1}
hold for Algorithm~\ref{alg:maxF(y):Y}.
The method in \cite{HenLas05} can be applied to extract maximizers in Step~3.
We can show that it must also terminate within finitely many
iterations, under some genericity conditions.

\begin{cond} \label{cond:F(x*):y}
The polynomial tuple $h$ is nonsingular and the point $x^*$
satisfies one (not necessarily all) of the following:
\bit

\item [(1)] $\mbox{IQ}(h_{eq},h_{in})$ is archimedean;

\item [(2)] the equation $\phi_{x^*}(y)=0$ has
finitely many real solutions;

\item [(3)] $\mbox{IQ}(\phi_{x^*}, \psi_{x^*})$ is archimedean.

\eit

\end{cond}

By the same argument as for Condition~\ref{cond:fix:y*},
we can also see that Condition~\ref{cond:F(x*):y} holds generically.
Similarly, Algorithm~\ref{alg:maxF(y):Y} also
terminates within finitely many iterations under some genericity conditions.

\begin{theorem}
\label{thm:max:F(x*y)}
Assume that \reff{minF(x*y):y:Y} has a maximizer
and Condition~\ref{cond:F(x*):y} holds.
If each maximizer of \reff{minF(x*y):y:Y} is an isolated critical point,
then, for all $k$ big enough, \reff{max:F(x*y):las}
has a maximizer and each of them must satisfy \reff{flat:F(x*y):Y}.
\end{theorem}

The proof of Theorem~\ref{thm:max:F(x*y)}
will be given in Section~\ref{sc:proof}. Similarly,
when $F,g,h$ are generic, each maximizer of
\reff{minF(x*y):y:Y} is an isolated
critical point of \reff{minF(x*y):y:Y}.

\section{Some proofs}
\label{sc:proof}

This section gives the proofs for some
theorems in the previous sections.

\begin{proof}[Proof of Theorem~\ref{thm:Cktpt:fnt}]
Under the genericity assumption, the polynomial tuples
$g,h$ are nonsingular, so the Lagrange multipliers
in \reff{opcd:F:x}-\reff{opcd:F:y}
can be expressed as in \reff{lmd(x,y)}-\reff{mu(x,y)}.
Hence, \reff{cpxkkt:Fxy} is equivalent to the
polynomial system in $(x,y,\lmd,\mu)$:
\be \label{cpxkkt:F:lmdmu}
\left\{ \baray{c}
\nabla_xF(x, y) = \sum_{ i \in \mc{E}^X }  \lmd_i \nabla_x g_i(x) , \\
\nabla_yF(x, y) = \sum_{j \in \mc{E}^Y } \mu_j   \nabla_y h_j(y), \\
 g_i(x)  = 0 \, (i \in \mc{E}^X_1), \,
 \lmd_i   g_i(x)  = 0 \, (i \in \mc{E}^X_2), \\
 h_j(y)  = 0 \, (j \in \mc{E}^Y_1), \,
 \mu_j    h_j(y)  = 0 \, (j \in \mc{E}^Y_2).
\earay\right.
\ee
Due to the complementarity conditions,
$g_i(x)  = 0$ or $\lmd_i  = 0$ for each $i \in \mc{E}^X_2$,
and $h_j(x)  = 0$ or $\mu_j = 0$ for each $j \in \mc{E}^Y_2$.
Note that if $g_i(x)  \ne 0$ then $\lmd_i = 0$ and if $h_j(x)  \ne 0$ then $\mu_j = 0$.
Since $\mc{E}^X_2, \mc{E}^Y_2$ are finite labeling sets,
there are only finitely many cases of $g_i(x)  = 0$ or $g_i(x)  \ne 0$,
$h_j(x)  = 0$ or $h_j(x)  \ne 0$.
We prove the conclusion is true for every case.
Moreover, if $g_i(x)  = 0$ for $i \in \mc{E}^X_2$,
then the inequality $g_i(x) \ge 0$ can be counted as an equality constraint.
The same is true for $h_j(x)  = 0$ with $j \in \mc{E}^Y_2$.
Therefore, we only need to prove the conclusion is true
for the case that has only equality constraints.
Without loss of generality, assume
$\mc{E}^X_2 =\mc{E}^Y_2 = \emptyset$ and write the labeling sets as
\[
\mc{E}^X_1 = \{1, \ldots, \ell_1\},
\qquad
\mc{E}^Y_1 = \{1, \ldots, \ell_2\}.
\]
When all $g_i$ are generic polynomials,
the equations $g_i(x) = 0$ ($i \in \mc{E}^X_1$) have no solutions if $\ell_1 > n$.
Similarly, the equations $h_j(x) = 0$ ($j \in \mc{E}^Y_1$)
have no solutions if $\ell_2 > m$ and all $h_j$ are generic.
Therefore, we only consider the case that $\ell_1 \leq n$ and $\ell_2 \leq m$.
When $F,g,h$ are generic,  we show that \reff{cpxkkt:F:lmdmu}
cannot have infinitely many solutions.
The system \reff{cpxkkt:F:lmdmu} is the same as
\be \label{cxkt:lamu:gh=0}
\left\{ \baray{c}
\nabla_xF(x, y) = \sum_{i=1}^{\ell_1}  \lmd_i \nabla_x g_i(x) , \,
g_1(x)  = \cdots = g_{\ell_1}(x) =  0, \\
\nabla_yF(x, y) = \sum_{j=1}^{\ell_2} \mu_j   \nabla_y h_j(y), \,
h_1(y)  = \cdots = h_{\ell_2}(y) = 0.
\earay\right.
\ee
Let $\tx = (x_0, x_1, \ldots, x_n)$
and $\ty = (y_0, y_1, \ldots, y_m)$.
Denote by $\tg_i(\tx)$ (resp., $\th_j(\ty)$)
the homogenization of $g_i(x)$ (resp., $h_j(y)$).
Let $\P^n$ denote the $n$-dimensional complex projective space.
Consider the projective variety
\[
\mc{U} \, := \,  \big \{ (\tx,\ty) \in \P^n \times \P^m:
\tg_i(\tx)  =  0 \, (i \in \mc{E}^X), \,
\th_j(\ty)  =  0 \, (j \in \mc{E}^Y) \big \}.
\]
It is smooth, by Bertini's theorem \cite{Har},
under the genericity assumption on $g_i, h_j$.
Denote the bi-homogenization of $F(x,y)$
\[
\tilde{F}(\tx,\ty) \, := \,
x_0^{a_0} y_0^{b_0}\tilde{F}(x/x_0, y/y_0).
\]
When $F(x,y)$ is generic, the projective variety
\[
\mc{V} := \mc{U} \cap \{ \tilde{F}(\tx,\ty) = 0 \}
\]
is also smooth. One can directly verify that
(for homogeneous polynomials)
\[
x^T\nabla_x \tilde{F}(\tx,\ty) + x_0 \pt_{x_0} \tilde{F}(\tx,\ty)
= a_0 \tilde{F}(\tx, \ty),
\]
\[
x^T\nabla_x \tg_i(\tx) + x_0 \pt_{x_0} \tg_i(\tx)
= a_i \tg_i(\tx),
\]
\[
y^T\nabla_y \tilde{F}(\tx,\ty) + y_0 \pt_{y_0} \tilde{F}(\tx,\ty)
= b_0 \tilde{F}(\tx, \ty),
\]
\[
y^T\nabla_y \th_j(\ty) + y_0 \pt_{y_0} \th_j(\ty)
= b_i \th_j(\ty).
\]
(They are called Euler's identities.)
Consider the determinantal variety
\[
W  := \left \{ (x,y) \in \cpx^n \times \cpx^m
\left| \baray{r}
\rank \bbm
\nabla_xF(x, y) &   \nabla_x g_1(x) & \cdots &  \nabla_x g_{\ell_1}(x)
\ebm \leq \ell_1   \\
\rank \bbm
\nabla_yF(x, y) &   \nabla_y h_1(y) & \cdots &  \nabla_y h_{\ell_2}(y)
\ebm \leq \ell_2
\earay \right. \right\}.
\]
Its homogenization is
\[
\widetilde{W}  := \left \{ (\tx,\ty) \in \P^n \times \P^m
\left| \baray{r}
\rank \bbm
\nabla_x \tilde{F}(\tx, \ty) &   \nabla_x \tg_1(\tx) & \cdots &  \nabla_x \tg_{\ell_1}(\tx)
\ebm \leq \ell_1    \\
\rank \bbm
\nabla_y \tilde{F}(\tx, \ty) &   \nabla_y \th_1(\ty) & \cdots &  \nabla_y \th_{\ell_2}(\ty)
\ebm \leq \ell_2
\earay \right. \right\}.
\]
The projectivization of \reff{cxkt:lamu:gh=0} is the intersection
\[
\widetilde{W} \cap \mc{U}.
\]
If \reff{cpxkkt:Fxy}
has infinitely many complex solutions,
so does \reff{cxkt:lamu:gh=0}. Then, $\widetilde{W} \cap \mc{U}$
must intersect the hypersurface
$\{ \tilde{F}(\tx,\ty) = 0 \}$.
This means that there exists $(\bx, \by) \in \mc{V}$ such that
\[
\nabla_x \tilde{F}(\bx,\by) = \sum_{i=1}^{\ell_1}  \lmd_i \nabla_x \tg_i(\bx), \quad
\nabla_y \tilde{F}(\bx,\by) = \sum_{j=1}^{\ell_2} \mu_j   \nabla_y \th_j(\by),
\]
for some $\lmd_i, \mu_j$. Also note
$\tg_i(\bx) = \th_j(\by) = \tilde{F}(\bx, \by) =0$. Write
\[
\bx = (\bx_0, \bx_1, \ldots, \bx_n), \quad
\by = (\by_0, \by_1, \ldots, \by_m).
\]
\bit

\item  If $\bx_0 \ne 0$ and $\by_0 \ne 0$,
by Euler's identities, we can further get
\[
\pt_{x_0} \tilde{F}(\bx,\by) = \sum_{i=1}^{\ell_1}  \lmd_i \pt_{x_0} \tg_i(\bx), \quad
\pt_{y_0} \tilde{F}(\bx,\by) = \sum_{j=1}^{\ell_2} \mu_j   \pt_{y_0} \th_j(\by).
\]
This implies that $\mc{V}$ is singular,
which is a contradiction.

\item If $x_0 = 0$ but $y_0 \ne 0$,
by Euler's identities, we can also get
\[
\pt_{y_0} \tilde{F}(\bx,\by) = \sum_{j=1}^{\ell_2} \mu_j   \pt_{y_0} \th_j(\by).
\]
This means the linear section
$\mc{V} \cap \{x_0 = 0\}$ is singular,
which is a contradiction again, by the genericity assumption
on $F,g,h$.

\item If $x_0 \ne 0$ but $y_0 = 0$, then we can have
\[
\pt_{x_0} \tilde{F}(\bx,\by) = \sum_{i=1}^{\ell_1}  \lmd_i \pt_{x_0} \tg_i(\bx).
\]
So the linear section $\mc{V} \cap \{y_0 = 0\}$ is singular,
which is again a contradiction.

\item If $x_0 = y_0 = 0$, then
$\mc{V} \cap \{x_0 = 0, y_0 = 0\}$
is singular. It is also
a contradiction, under the genericity assumption on $F,g,h$.

\eit
For every case, we obtain a contradiction.
Therefore, the polynomial system~\reff{cpxkkt:Fxy}
must have only finitely many complex solutions,
when $F,g,h$ are generic.
\end{proof}

\begin{proof}[Proof of Theorem~\ref{thm:iter:bound}]
Each solution of \reff{cpxkkt:Fxy}
is a critical point of $F(x,y)$
over the set $X \times Y$.
We count the number of critical points
by enumerating all possibilities of active constraints.
For an active labeling set
$\{i_1, \ldots, i_{r_1} \} \subseteq [\ell_1]$ (for $X$)
and an active labeling set
$\{j_1, \ldots, j_{r_2} \} \subseteq [\ell_2]$ (for $Y$),
an upper bound for the number is critical points is
$
a_{i_1} \cdots a_{ i_{r_1} } b_{j_1} \cdots b_{ j_{r_2} } \cdot s,
$
which is given by Theorem~2.2 of \cite{NR09}.
Summing this upper bound for all possible active constraints,
we eventually get the bound $M$.
Since $\mc{K}_0$ is a subset of \reff{cpxkkt:Fxy},
Algorithm~\ref{alg:pop:kt} must terminate
within $M$ iterations, for generic polynomials.
\end{proof}

\begin{proof}[Proof of Theorem~\ref{thm:Lascvg:xy}]
In Condition~\ref{cond:AC:FXYuv}, the item (1) or (2) implies (3).
Note that the dual optimization problem of \reff{<F,w>:lasrlx:ordk} is
\be  \label{maxgm:sos:ordk}
\left\{ \baray{cl}
\max & \gamma \\
\mbox{subject to} &  F - \gamma \in \mbox{IQ}(\phi,\psi)_{2k}.
\earay \right.
\ee

i) When \reff{optF:KKT:xy} is infeasible, the set
$\{\phi(x,y)=0, \, \psi(x,y) \geq 0 \}$ is empty. Since
$\mbox{IQ}(\phi,\psi)$ is archimedean,
by the classical Positivstellensatz~\cite{BCR}
and Putinar's Positivstellensatz~\cite{Put}, we have
$
-1 \in  \mbox{IQ}(\phi,\psi).
$
So,
$
-1 \in  \mbox{IQ}(\phi,\psi)_{2k}
$
for all such $k$ big enough. Hence, \reff{maxgm:sos:ordk}
is unbounded from above for all big $k$.
By weak duality, we know \reff{<F,w>:lasrlx:ordk}
must be infeasible.

ii) When \reff{optF:KKT:xy} is feasible,
every feasible point is a critical point. By Lemma~3.3 of \cite{DNP},
$F(x,y)$ achieves finitely many values on $\phi(x,y)=0$, say,
\[
c_1 < c_2 < \cdots < c_N.
\]
Recall that $f_{\ast}$ is the minimum value of \reff{minF(xy):phi:psi}.
So, $f_{\ast}$ is one of the $c_i$, say, $c_\ell = f_{\ast}$.
Since \reff{optF:KKT:xy} has only finitely many minimizers,
we can list them as the set
\[
O := \{ (u_1, v_1), \ldots, (u_B, v_B) \}.
\]
If $(x,y)$ is a feasible point of \reff{optF:KKT:xy},
then either $F(x,y) = c_k$ with $k>\ell$, or
$(x,y)$ is one of $(u_1, v_1), \ldots, (u_B, v_B)$.
Define the polynomial
\[
P(x, y) := \Big( \prod_{i = \ell+1}^N ( F(x,y)- c_i)^2 \Big) \cdot
\Bigg( \prod_{(u_j,v_j) \in O }  \Big( \|x-u_j\|^2 + \| y-v_j \|^2 \Big) \Bigg).
\]
We partition the set $\{ \phi(x,y) = 0 \}$
into four disjoint ones:
\[
\baray{l}
U_1 := \left\{  \phi(x,y) =0,
c_1 \leq F(x,y) \leq  c_{\ell-1}  \right\}, \\
U_2 := \left\{  \phi(x,y) =0,  F(x,y) = c_{\ell}, (x,y) \not\in O \right\}, \\
U_3 := \left\{  \phi(x,y) =0,  F(x,y) = c_{\ell}, (x,y) \in O  \right\}, \\
U_4 := \left\{ \phi(x,y) =0,
c_{\ell+1} \leq F(x,y) \leq  c_N  \right\}.
\earay
\]
Note that $U_3$ is the set of minimizers for \reff{minF(xy):phi:psi}.
\bit

\item For all $(x,y) \in U_1$ and $i=\ell+1, \ldots, N$,
\[
(F(x,y)-c_i)^2 \geq (c_{\ell-1} - c_{\ell+1} )^2.
\]
The set $U_1$ is closed and each $(u_j, v_j) \not \in U_1$.
The distance from $(u_j, v_j)$ to $U_1$ is positive.
Hence, there exists $\eps_1 >0$ such that
$P(x,y) > \eps_1$ for all $(x,y) \in U_1$.

\item  For all $(x,y) \in U_2$,
$
(F(x,y)-c_i)^2 = (c_{\ell} - c_i)^2.
$
For each $(u_j, v_j) \in O$,
its distance to $U_2$ is positive.
This is because each $(u_i, v_i) \in O$
is an isolated real critical point.
So, there exists $\eps_2 >0$ such that
$P(x,y) > \eps_2$ for all $(x,y) \in U_2$.

\eit
Denote the new polynomial
\[
q(x,y) := \min(\eps_1, \eps_2) - P(x,y).
\]
On the set $\{ \phi(x,y)=0 \}$, the inequality
$q(x,y) \geq 0$ implies $(x,y) \in U_3 \cup U_4$.
Therefore, \reff{optF:KKT:xy} is equivalent to the optimization problem
\be \label{opt:q>=0}
\left\{ \baray{cl}
\min\limits_{x,y} & F(x,y) \\
\mbox{subject to} &  \phi(x,y) = 0, \, q(x,y) \geq 0.
\earay \right.
\ee
Note that $q(x,y) > 0$ on the feasible set of \reff{optF:KKT:xy}.
\big(This is because if $(x,y)$ is a feasible point of \reff{optF:KKT:xy},
then $F(x,y) \geq f_{\ast} = c_{\ell}$, so $(x,y) \not\in U_1$.
If $F(x,y)= c_{\ell}$, then $(x,y) \in O$ and $P(x,y)=0$,
so $q(x,y) = \min(\eps_1, \eps_2) > 0$.
If $F(x,y) > c_{\ell}$, then $P(x,y) =0$ and we also have
$q(x,y) = \min(\eps_1, \eps_2) > 0$.\big)
By Condition~\ref{cond:AC:FXYuv} and Putinar's Positivstellensatz,
it holds that
$
q \, \in \, \mbox{IQ}(\phi,\psi).
$
Now, we consider the hierarchy of
Lasserre's relaxations for solving \reff{opt:q>=0}:
\be \label{mom:k:rho}
\left\{\baray{rl}
f_k^\prm  := \min & \langle F, w \rangle \\
\mbox{subject to} & (w)_0 = 1,  M_k(w) \succeq 0, \\
      & L_{\phi}^{(k)}(w) = 0,   \,  L_{q}^{(k)}(w) \succeq 0.
\earay \right.
\ee
Its dual optimization problem is
\be  \label{sos:k:rho}
\left\{\baray{rl}
f_k := \, \max & \gamma  \\
\mbox{subject to} &  F-\gamma \in \mbox{IQ}(\phi, q)_{2k}.
\earay \right.
\ee

\noindent
{\bf Claim:}
For all $k$ big enough, it holds that
$
f_k =  f_k^\prm = f_{\ast}.
$
\begin{proof}
The possible objective values of \reff{opt:q>=0}
are $c_\ell, \ldots, c_N$.
Let $p_1, \ldots,p_N$
be real univariate polynomials such that
$p_i( c_j) = 0$ for $i \ne j$ and $p_i( c_j) = 1$ for $i = j$.
Let
\[
s_i :=  ( c_i -  f_{\ast}) \big( p_i (  F)  \big)^2, \quad
i=\ell, \ldots, N.
\]
Then $s := s_\ell + \cdots + s_N \in \Sig[x]_{2k_1}$
for some order $k_1 > 0$. Let
\[
\hat{F}:= F - f_{\ast} - s.
\]
Note that $\hat{F}(x) \equiv 0$ on the set
\[
\mc{K}_2 \, := \, \{
\phi(x,y) = 0, \, q(x,y) \geq 0 \}.
\]
It has a single inequality.
By the Positivstellensatz~\cite[Corollary~4.1.8]{BCR},
there exist $0< t \in \N$ and
$Q = b_0 + q b_1 $ ($b_0, b_1 \in \Sig[x]$) such that
$
\hat{F}^{2 t} + Q \in \mbox{Ideal}(\phi).
$
Note that $Q \in \mbox{Qmod}(q)$.
For all $\eps >0$ and $\tau >0$, we have
$
\hat{F} + \eps = \phi_\eps + \theta_\eps
$
where
\[
\phi_\eps = -\tau \eps^{1-2t}
\big(\hat{F}^{2t} + Q \big),
\]
\[
\theta_\eps = \eps \Big(1 + \hat{F}/\eps +
\tau ( \hat{F}/\eps)^{2t} \Big)
+ \tau \eps^{1-2t} Q.
\]
By Lemma~2.1 of \cite{POPreal}, when $\tau \geq \frac{1}{2t}$,
there exists $k_2$ such that, for all $\eps >0$,
\[
\phi_\eps \in \mbox{Ideal}(\phi)_{2k_2}, \quad
\theta_\eps \in \mbox{Qmod}(q)_{2k_2}.
\]
Hence, we can get
\[
F - (f_{\ast} -\eps) = \phi_\eps + \sig_\eps,
\]
where $\sig_\eps = \theta_\eps + s \in \mbox{Qmod}(q)_{2k_2}$
for all $\eps >0$. For all $\eps>0$,
$\gamma = f_{\ast}-\eps$ is feasible in \reff{sos:k:rho} for the order $k_2$,
so $f_{k_2} \geq f_{\ast}$.
Because $f_k \leq f_{k+1} \leq \cdots \leq f_{\ast}$,
we have $f_{k} = f_{k}^\prm = f_{\ast}$ for all $k \geq k_2$.
\end{proof}

Because $q \in  \mbox{Qmod}(\psi)$,
each $w$, which is feasible for \reff{<F,w>:lasrlx:ordk},
is also feasible for \reff{mom:k:rho}.
This can be implied by \cite[Lemma~2.5]{Nie-ft}.
So, when $k$ is big, each $w$ is also a minimizer of \reff{mom:k:rho}.
The problem \reff{opt:q>=0} also has only finitely many minimizers.
By Theorem~2.6 of \cite{Nie-ft},
the condition \reff{cond:flat} must be satisfied
for some $t\in [d_0,k]$, when $k$ is big enough.
\end{proof}

\begin{proof}[Proof of Theorem~\ref{thm:min:F(xy*)}]
The proof is the same as the one
for Theorem~\ref{thm:Lascvg:xy}.
This is because the Lasserre's relaxations
\reff{min<F(xy*),z>} are constructed by using
optimality conditions of \reff{minF(xy*):X},
which is the same as for Theorem~\ref{thm:Lascvg:xy}.
In other words, Theorem~\ref{thm:min:F(xy*)}
can be thought of a special version of Theorem~\ref{thm:Lascvg:xy}
with $K_1 = K_2 = \emptyset$, without variable $y$.
The assumptions are the same.
Therefore, the same proof can be used.
\end{proof}

\begin{proof}[Proof of Theorem~\ref{thm:max:F(x*y)}]
The proof is the same as the one
for Theorem~\ref{thm:min:F(xy*)}.
\end{proof}

\section{Numerical Experiments}
\label{sc:num}

This section presents numerical examples of applying
Algorithm~\ref{alg:pop:kt} to solve saddle point problems.
The computation is implemented in MATLAB R2012a,
on a Lenovo Laptop with CPU@2.90GHz and RAM 16.0G.
The Lasserre type moment semidefinite relaxations
are solved by the software {\tt GloptiPoly~3} \cite{GloPol3},
which calls the semidefinite program solver {\tt SeDuMi} \cite{Sturm}.
For cleanness, only four decimal digits are displayed
for computational results.

In prior existing references,
there are very few examples of non convex-concave type SPPPs.
We construct various examples, with different types of functions and constraints.
When $g,h$ are nonsingular tuples, the Lagrange multipliers
$\lmd(x,y),\mu(x,y)$ can be expressed by polynomials as in
\reff{lmd(x,y)}-\reff{mu(x,y)}.
Here we give some expressions for $\lmd(x,y)$
that will be frequently used in the examples.
The expressions are similar for $\mu(x,y)$.
Let $F(x,y)$ be the objective.

\begin{itemize}

\item For the simplex
$\Delta_n =\{x \in \re^n: e^Tx=1, x \geq 0 \}$,
$g=(e^Tx-1, x_1, \ldots, x_n)$ and we have
\be \label{lmd(x,y):simplex}
\lmd(x,y) \,= \, (x^T \nabla_x F, F_{x_1}-x^T\nabla_x F,
\ldots, F_{x_n}-x^T\nabla_x F).
\ee

\item For the hypercube set $[-1,1]^n$,
$g=(1-x_1^2, \ldots, 1-x_n^2)$ and
\be  \label{lmd(x,y):hypercube}
\lambda(x,y) \,= \, -\frac{1}{2} ( x_1 F_{x_1}, \ldots, x_n F_{x_n}).
\ee

\item For the box constraint $[0,1]^n$,
$g=(x_1, \ldots, x_n, 1-x_1, \ldots, 1-x_n)$ and
\be \label{lmd(x,y):box}
\lambda(x,y) \,= \, ((1-x_1) F_{x_1}, \ldots, (1-x_n) F_{x_n},
-x_1 F_{x_1}, \ldots, -x_n F_{x_n} ).
\ee

\item For the ball $B_n(0,1) = \{ x \in \re^n: \, \|x\| \leq 1\}$
or sphere $\mathbb{S}^{n-1} = \{  x \in \re^n: \,  \|x\| = 1\}$,
$g=1-x^Tx$ and we have
\be \label{lmd(x,y):ball}
\lambda(x,y) \,= \, -\frac{1}{2}x^T \nabla_x F.
\ee

\item For the nonnegative orthant $\re_+^n$,
$g=(x_1,\ldots,x_n)$ and we have
\be  \label{lmd(x,y):nngorthant}
\lambda(x,y) \, = \, (F_{x_1}, \ldots,  F_{x_n}).
\ee

\end{itemize}

We refer to \cite{LagExp} for more details about Lagrange multiplier expressions.

\subsection{Some explicit examples}

\begin{exm} \label{exm6.1:simplex}
Consider the simplex feasible sets $X = \Dt_n$, $Y = \Dt_m$.
The Lagrange multipliers can be expressed as in \reff{lmd(x,y):simplex}. \\
(i) Let $n=m=3$ and
\[
F(x,y)  \, = \,
x_1x_2+x_2x_3+x_3y_1+x_1y_3+y_1y_2+y_2y_3.
\]
This function is neither convex in $x$ nor concave in $y$.
After $1$ iteration by Algorithm~\ref{alg:pop:kt}, we got the saddle point:
\[
x^* =(0.0000, \, 1.0000, \, 0.0000), \quad
y^*=(0.2500,\, 0.5000,\, 0.2500).
\]
It took about $2$ seconds. \\
(ii) Let $n=m=3$ and $F(x,y)$ be the function
\begin{multline*}
x_1^3+x_2^3-x_3^3-y_1^3-y_2^3+y_3^3+x_3y_1y_2(y_1+y_2)+x_2y_1y_3(y_1+y_3)+x_1y_2y_3(y_2+y_3).
\end{multline*}
This function is neither convex in $x$ nor concave in $y$.
After $2$ iterations  by Algorithm~\ref{alg:pop:kt}, we got the saddle point
\[
x^*=(0.0000,0.0000,1.0000), \quad y^*=(0.0000,0.0000,1.0000).
\]
It took about $7.5$ seconds. \\
(iii) Let $n=m=4$ and
\[
F(x,y) \,=\, {\sum}_{i,j=1}^4  x_i^2y_j^2-{\sum}_{i\neq j} (x_ix_j+y_iy_j).
\]
This function is neither convex in $x$ nor concave in $y$.
After $2$ iterations by Algorithm~\ref{alg:pop:kt}, we got $4$ saddle points:
\[
x^*=(0.2500,0.2500,0.2500,0.2500), \quad y^*=e_i,
\]
with $i=1,2,3,4$.
It took about $99$ seconds. \\
(iv) Let $n=m=3$ and
\[
F(x,y)  \, := \,
x_1x_2y_1y_2  + x_2x_3y_2y_3 + x_3x_1y_3y_1
-x_1^2y_3^2 -x_2^2y_1^2 -x_3^2y_2^2.
\]
This function is neither convex in $x$ nor concave in $y$.
After $4$ iterations  by Algorithm~\ref{alg:pop:kt},
we got that there is no saddle point.
It took about $32$ seconds.
\end{exm}

\begin{exm}  \label{exm6.2:box}
Consider the box constraints $X =[0,1]^n$ and $Y =[0, 1]^m$.
The Lagrange multipliers can be expressed as in \reff{lmd(x,y):box}. \\
(i) Consider $n=m=2$ and
\[
F(x,y)  \, := \,
 (x_1+ x_2 + y_1 + y_2 + 1)^2
-4(x_1x_2 +x_2y_1 + y_1y_2 + y_2 + x_1).
\]
This function is convex in $x$ but not concave in $y$.
After $2$ iterations by Algorithm~\ref{alg:pop:kt}, we got the saddle point
\[
x^* = ( 0.3249,    0.3249 ), \quad
y^* = ( 1.0000,    0.0000 ).
\]
It took about $3.7$ seconds. \\
(ii) Let $n=m=3$ and
\[
F(x,y) \, := \, {\sum}_{i=1}^n (x_i+y_i) + {\sum}_{i<j} (x_i^2y_j^2-y_i^2x_j^2) .
\]
This function is neither convex in $x$ nor concave in $y$.
After $3$ iterations by Algorithm~\ref{alg:pop:kt}, we got that
there is no saddle point.
It took about $12.8$ seconds.
\end{exm}

\begin{exm}
Consider the cube constraints $X=Y=[-1,1]^3$.
The Lagrange multipliers can be expressed as in \reff{lmd(x,y):hypercube}. \\
(i) Consider the function
\[
F(x,y)  \, := \,
{\sum}_{i=1}^3  (x_i+y_i)  - {\prod}_{i=1}^3 (x_i - y_i).
\]
This function is neither convex in $x$ nor concave in $y$.
After $1$ iteration by Algorithm~\ref{alg:pop:kt}, we got $3$ saddle points:
\[
x^* = (-1.0000,   -1.0000,    1.0000), \quad
y^* = (1.0000,    1.0000,    1.0000),
\]
\[
x^* = ( -1.0000,    1.0000,   -1.0000 ), \quad
y^* = (1.0000,    1.0000,    1.0000),
\]
\[
x^* = (1.0000,   -1.0000,   -1.0000), \quad
y^* = (1.0000,    1.0000,    1.0000).
\]
It took about $75$ seconds. \\
(ii) Consider the function
\[
F(x,y)  \, := \, y^Ty-x^Tx+{\sum}_{1 \le i<j \le 3} (x_iy_j-x_jy_i).
\]
This function is neither convex in $x$ nor concave in $y$.
After $4$ iterations by Algorithm~\ref{alg:pop:kt}, we got the saddle point
\[
x^* =  (-1.0000,    1.0000,   -1.0000), \quad
y^* =  (-1.0000,    1.0000,   -1.0000).
\]
It took about $6$ seconds.
\end{exm}

\begin{exm}
Consider the sphere constraints
$X = \mathbb{S}^{2}$ and $Y = \mathbb{S}^{2}$. They are not convex.
The Lagrange multipliers can be expressed as in \reff{lmd(x,y):ball}. \\
(i) Let $F(x,y)$ be the function
\[
x_1^3 + x_2^3 + x_3^3 + y_1^3 + y_2^3 + y_3^3
+2(x_1x_2y_1y_2 + x_1x_3y_1y_3 + x_2x_3y_2y_3).
\]
After $2$ iterations by Algorithm~\ref{alg:pop:kt}, we got $9$ saddle points
$(-e_i, e_j)$, with $i, j =1, 2, 3$.
It took about $64$ seconds. \\
(ii) Let $F(x,y)$ be the function
\begin{multline*}
x_1^2y_1^2+x_2^2y_2^2+x_3^2y_3^2 +
x_1^2 y_2y_3 + x_2^2 y_1y_3 + x_3^2 y_1y_2 +
y_1^2x_2x_3 + y_2^2x_1x_3 + y_3^2x_1x_2.
\end{multline*}
After $4$ iterations by Algorithm~\ref{alg:pop:kt},
we got that there is no saddle point.
It took about $127$ seconds.
\end{exm}

\begin{exm}
Let $X=Y=B_3(0,1)$ be the ball constraints and
\[
F(x,y) \,:= \, x_1^2y_1+2x_2^2y_2+3x_3^2y_3-x_1-x_2-x_3 .
\]
The Lagrange multipliers can be expressed as in \reff{lmd(x,y):ball}.
The function $F$ is not convex in $x$ but is concave in $y$.
After $1$ iteration by Algorithm~\ref{alg:pop:kt}, we got the saddle point:
\[
x^*=(0.7264,0.4576,0.3492),\quad
y^*=( 0.6883,0.5463, 0.4772).
\]
It took about $3.3$ seconds.
\end{exm}

\begin{exm}
Consider the function
\[
F(x,y) \,:=\,  x_1^2y_2y_3+y_1^2x_2x_3+x_2^2y_1y_3+y_2^2x_1x_3
+x_3^2y_1y_2+y_3^2x_1x_2
\]
and the sets
\[
X \,:=\, \{x\in \re^3: \, x^Tx-1 =0, x\ge 0\}, \quad
Y\, := \, \{y \in \re^3: \, y^Ty-1=0, y\ge 0\}.
\]
They are nonnegative portions of spheres.
The feasible sets $X,Y$ are non-convex.
The Lagrange multipliers are expressed as
\[
	\lambda(x,y) = (\frac{1}{2}x^T\nabla_x F, F_{x_1}-x_1x^T\nabla_xF,F_{x_2}-x_2x^T\nabla_xF, F_{x_3}-x_3x^T\nabla_xF ),
\]
\[
	\mu(x,y) = (\frac{1}{2}y^T\nabla_y F, F_{y_1}-y_1y^T\nabla_yF,F_{y_2}-y_2y^T\nabla_yF, F_{y_3}-y_3y^T\nabla_yF ) .
\]
After $3$ iterations by Algorithm~\ref{alg:pop:kt}, we got that
there is no saddle point.
It took about $37.3$ seconds.
\end{exm}

\begin{exm}
Let $X=Y= \re_+^4$ be the nonnegative orthant and $F(x,y)$ be
\begin{multline*}
y_1(x_2+x_3+x_4-1)^2+y_2(x_1+x_3+x_4-2)^2+y_3(x_1+x_2+x_4-3)^2 \\
-y_4(x_1+x_2+x_3-4)^2-\Big(x_1(y_2+y_3+y_4-1)^2+x_2(y_1+y_3+y_4-2)^2 \\
   -x_3(y_1+y_2+y_4-3)^2+x_4(y_1+y_2+y_3-4)^2\Big).
\end{multline*}
The Lagrange multipliers can be expressed as in \reff{lmd(x,y):nngorthant}.
The function $F$ is neither convex in $x$ nor concave in $y$.
After $1$ iteration by Algorithm~\ref{alg:pop:kt}, we got the saddle point
\[
 x^*=(1.5075, 0.5337,    0.0000,    0.5018 ), \quad
 y^*=(2.4143,    1.1463,    0.0000,    0.0000).
\]
It took about $4.8$ seconds.
\end{exm}

\begin{exm}
Let $X = Y = \re^3$ be the entire space, i.e., there are no constraints.
There are no needs for Lagrange multiplier expressions.
Consider the function
\[
F(x,y)=\sum_{i=1}^3(x_i^4-y_i^4+x_i+y_i)+\sum_{i\neq j}x_i^3y_j^3 .
\]
It is neither convex in $x$ nor concave in $y$.
After $1$ iteration by Algorithm~\ref{alg:pop:kt}, we got the saddle point
\[
x^*= -(0.6981,   0.6981,   0.6981), \quad
y^*= (0.4979,    0.4979,    0.4979).
\]
It took about $113$ seconds.
\end{exm}

\begin{exm}
Consider the sets and the function
\[
X \,:= \, \{x\in \mathbb{R}^3:\, x_1\ge 0,x_1x_2\ge 1,x_2x_3\ge 1\},
\]
\[
Y \, := \, \{y\in \mathbb{R}^3:\, y_1\ge 0,y_1y_2\ge 1,y_2y_3\ge 1\},
\]
\[
F(x,y) \,:=\, x_1^3y_1+x_2^3y_2+x_3^3y_3-3x_1x_2x_3-y_1^2-2y_2^2-3y_3^2 .
\]
The function $F(x,y)$ is not convex in $x$ but is concave in $y$.
The Lagrange multipliers can be expressed as
\[
\lambda_1 = (1-x_1x_2)F_{x_1}, \quad
\lambda_2 = x_1F_{x_1},\quad
\lambda_3 =  -x_1F_{x_1}+x_2F_{x_2}.
\]
The same expressions are for $\mu_j(x,y)$.
After $9$ iterations by Algorithm~\ref{alg:pop:kt}, we get the saddle point:
\[
 x^*=(1.2599,1.2181,1.3032),\quad y^*=(1.0000,1.1067,0.9036).
\]
It took about $64$ seconds.
\end{exm}

\subsection{Some application problems}

\begin{example} \label{ex:zero sum}
We consider the saddle point problem arising from zero sum games with two players.
Suppose $x \in \re^n$ is the strategy for the first player
and $y \in \re^m$ is the strategy for the second one.
The usual constraints for strategies are given by simplices,
which represent probability measures on finite sets.
So we consider feasible sets $X = \Dt_n$, $Y = \Dt_m$.
Suppose the profit function of the first player is
\[
f_1(x,y) \, = \, x^TA_1x + y^TA_2y + x^TBy ,
\]
for matrices $A_1 \in \re^{n\times n}$, $A_2 \in \re^{m\times m}$, $B \in \re^{n\times m}$.
For the zero sum game, the profit function for the second player is
$f_2(x,y) := -f_1(x,y)$. Each player wants to maximize the profit,
for the given strategy of the other player.
The Nash equilibrium is a point $(x^*,y^*)$ such that
the maximum of $f_1(x,y^*) $ over $\Dt_n$ is achieved at $x^*$,
while the maximum of $f_2(x^*,y)$ over $\Dt_m$ is achieved at $y^*$.
This is equivalent to that $(x^*, y^*)$ is a saddle point of the function
$F := -f_1(x,y)$ over $X, Y$.
For instance, we consider the matrices
\[
A_1 =  \begin{pmatrix*}[r]
-4 & 4 & 0 & 3 & -4\\ 3 & 4 & 3 & -4 & -5\\ -3 & 0 & -2 & 0 & 4\\
-4 & -4 & -1 & 3 & -5\\ 4 & 1 & -3 & 0 & -5 \end{pmatrix*}, \quad
A_2 = \begin{pmatrix*}[r]
-4 & 4 & 1 & 0 & 1\\ -2 & -4 & 2 & -3 & 1\\ -3 & 1 & 1 & 4 & 4\\
3 & -4 & 0 & 1 & -2\\ -1 & -3 & -1 & 3 & -2
\end{pmatrix*},
\]

\[
B = \begin{pmatrix*}[r]
-2 & -4 & -2 & -5 & 3\\ 0 & 0 & 2 & 4 & 2\\ 0 & -4 & -1 & -5 & 3\\
1 & -3 & -4 & 0 & -3\\ 3 & -1 & -5 & 4 & -4
\end{pmatrix*} .
\]
The resulting saddle point problem is of the non convex-concave type.
After $2$ iterations by Algorithm~\ref{alg:pop:kt}, we get two Nash equilibria
\[
x^* = (0,1,0,0,0),\quad y^*=(1,0,0,0,0) ,
\]
\[
x^* = (0,1,0,0,0),\quad y^*=(0,1,0,0,0).
\]
It took about $7$ seconds.
\end{example}

\begin{example} \label{ex:robust}
Consider the portfolio optimization problem \cite{hall03,zhu09}
\[
\min_{x\in X} \quad  -\mu^Tx + x^TQx,
\]
where $Q$ is a covariance matrix and $\mu$ is the estimation of some parameters.
There often exists a perturbation $(\dt \mu, \dt Q)$ for $(\mu,Q)$.
This results in two types of robust optimization problems
\[
\baray{ccl}
\min\limits_{x\in X} & \max\limits_{(\delta \mu,\delta Q)\in Y} &
-(\mu+\delta \mu)^Tx + x^T(Q+\delta Q)x,
\\
\max\limits_{(\delta \mu,\delta Q)\in Y} & \min\limits_{x\in X} &
-(\mu+\delta \mu)^Tx + x^T(Q+\delta Q)x .
\earay
\]
We look for $x^*$ and $(\dt \mu^*, \dt Q^*)$
that can solve the above two robust optimization problems simultaneously.
This is equivalent to the saddle point problem with
$F = -(\mu+\delta \mu)^Tx + x^T(Q+\delta Q)x$.
For instance, consider the case that
\[
Q = \begin{pmatrix*}[r]
5 & -4 & -2\\ -4 & 13 & 10\\ -2 & 10 & 8
\end{pmatrix*},\,
\mu = \begin{pmatrix*}[r]
0\\ -1\\ 3
\end{pmatrix*} ,
\]
with the feasible sets
\[
X \,:= \, \{x \in \re^3 \mid  -0.5 \le x_i \le 0.5,\, i=1,\ldots,n\} ,
\]
\[
Y \, := \,  \left \{ (\dt \mu, \dt Q) \in  \re^3 \times
\mathcal{S}\re^{3 \times 3} \Big|
\baray{c}
-0.1 \le (\delta \mu)_k, \, (\delta Q)_{ij} \le 0.1, \\
 1 \le k \le 3, \, 1\le i, j\le 3
\earay
\right \} .
\]
In the above, $\mathcal{S}\re^{3 \times 3}$ denotes the space of
real symmetric $3$-by-$3$ matrices.
The Lagrange multipliers can be similarly expressed as in \reff{lmd(x,y):box}.
After 1 iteration by Algorithm~\ref{alg:pop:kt}, we got the saddle point
\[
x^* =
\begin{pmatrix*}[r]
-0.1289\\
-0.4506 \\
    0.5000
\end{pmatrix*},\,
\delta Q^* =
\begin{pmatrix*}[r]
0.1 & 0.1 & -0.1 \\
0.1 & 0.1 & -0.1 \\
-0.1 & -0.1 & 0.1
\end{pmatrix*},\,
\delta \mu^* =
\begin{pmatrix*}[r]
0.1 \\ 0.1 \\ -0.1
\end{pmatrix*}.
\]
It took about $32$ seconds. The above two min-max and max-min
optimization problems are solved simultaneously by them.
\end{example}

\subsection{Some comparisons with other methods}
\label{ssc:compa}

Upon the request by referees, we give some comparisons between
Algorithm~\ref{alg:pop:kt} and other methods.
The saddle point problems can also be solved by the straightforward approach
of enumerating all KKT points.
When all KKT points are to be computed,
the numerical homotopy method such as {\tt Bertini} \cite{BHSW06} can be used.
Saddle points can also be computed by quantifier elimination (QE) methods.
Note that the definition \reff{dfsad:F:XY} automatically gives
a quantifier formula for saddle points.
In the following, we give a comparison
of the performance of these methods and Algorithm~\ref{alg:pop:kt}.

For computing all KKT points, the {\tt Maple} function \texttt{Solve}
is used. After they are obtained,
we select saddle points from them by checking the definition.
For the QE approach, the {\tt Maple} function \texttt{QuantifierElimination}
is used to solve the quantifier elimination formulae.
For the numerical homotopy approach, the software \texttt{Bertini}
is used to solve the KKT system for getting all KKT points first
and then we select saddle points from them.
When there are infinitely many KKT points,
the function \texttt{Solve} and the software \texttt{Bertini} experience
difficulty to get saddle points by enumerating KKT points.
This happens for Examples 6.1(i), 6.2(i), 6.4(ii).
The computational time (in seconds or hours) for these methods
is reported in Table~\ref{tab:compa:KKTQE}.
We would like to remark that the {\tt Maple} function \texttt{QuantifierElimination}
cannot solve any example question
(it does not terminate within $6$ hours for each one).
So we also try the {\tt Mathematica} function {\tt Resolve}
to implement the QE method. It can solve Example~6.1(iii)
in about $1$ second, but it cannot solve any other example question
(it does not terminate within $6$ hours).
The software {\tt Bertini} can solve
Examples~6.1(ii), 6.2(ii), 6.3(i), 6.4(i),
6.5, 6.7, 6.8, 6.9.  For other example questions,
it cannot finish within $6$ hours.
For these examples, Algorithm~\ref{alg:pop:kt} took much less computational time,
except Example~6.1(iii).
We also like to remark that the symbolic methods like
QE can obtain saddle points exactly in symbolic operations,
while numerical methods can only obtain saddle points correctly up to round off errors.

\begin{table}
\caption{Comparisons with other types of methods}
\label{tab:compa:KKTQE}
\begin{center}
\begin{tabular}{l|l|l|l|l}
Exemp. & Alg.~\ref{alg:pop:kt} & KKT ({\tt Solve}) & QE & {\tt Bertini}   \\   \hline
6.1(i) & $2$ sec. &  $\infty$ KKT points & $>\,6$ hours  &  $\infty$ KKT points  \\  \hline
6.1(ii) &  $7.5$ sec. &   $>\,6$ hours &  $>\,6$ hours  & $380$ sec. \\   \hline
6.1(iii) & $99$ sec. &  $190$ sec. &  $1$ sec.  & $>\,6$ hours \\  \hline
6.1(iv) & $32$ sec. &  $>\,6$ hours & $>\,6$ hours  & $>\,6$ hours \\  \hline
6.2(i) &  $3.7$ sec. &  $\infty$ KKT points &  $>\,6$ hours & $\infty$ KKT points \\  \hline
6.2(ii) & $12.8$ sec. &  $>\,6$ hours & $>\,6$ hours  & $780$ sec. \\  \hline
6.3(i) & $75$ sec. &  $30$ sec. &  $>\,6$ hours & 1100 sec.\\    \hline
6.3(ii) & $6$ sec. &  $13011$ sec. &  $>\,6$ hours & $>\,6$ hours \\    \hline
6.4(i) & $64$ sec. &  $13921$ sec. &  $>\,6$ hours & $441$ sec. \\  \hline
6.4(ii) & $127$ sec. & $\infty$ KKT points &  $>\,6$ hours &  $\infty$ KKT points \\  \hline
6.5 &  $3.3$ sec. & $>\,6$ hours &  $>\,6$ hours & $8701$ sec. \\  \hline
6.6 &  $37.3$ sec. & $>\,6$ hours &  $>\,6$ hours & $>\,6$ hours \\  \hline
6.7 & $4.8$ sec. & $>\,6$ hours &  $>\,6$ hours & $78$ sec. \\   \hline
6.8 & $113$ sec.  & $>\,6$ hours &  $>\,6$ hours & $102$ sec. \\  \hline
6.9 & $64$ sec.  & $>\,6$ hours &  $>\,6$ hours & $6293$ sec. \\  \hline
6.10 & $7$ sec.  & $>\,6$ hours &  $>\,6$ hours & $>\,6$ hours \\ \hline
6.11 & $32$ sec.  & $>\,6$ hours &  $>\,6$ hours & $>\,6$ hours \\
\end{tabular}
\end{center}

\end{table}

\section{Conclusions and discussions}
\label{sc:con}

This paper discusses how to solve the saddle point problem of polynomials.
We propose an algorithm
(i.e., Algorithm~\ref{alg:pop:kt}) for computing saddle points.
The Lasserre type semidefinite relaxations are used to solve
the polynomial optimization problems. Under some genericity assumptions,
the proposed algorithm can compute a saddle point
if there exists one. If there does not exist a saddle point,
the algorithm can detect the nonexistence.
However, we would like to remark that
Algorithm~\ref{alg:pop:kt} can always be applied,
no matter whether the defining polynomials are generic or not.
The algorithm needs to solve semidefinite programs
for Lasserre type relaxations. Since semidefinite programs
are usually solved numerically (e.g., by {\tt SeDuMi}),
the computed solutions by Algorithm~\ref{alg:pop:kt}
are correct up to numerical errors.
If the computed solutions are not accurate enough,
classical nonlinear optimization methods can be applied to improve the accuracy.
The method given in this paper can be used
to solve saddle point problems from broad applications,
such as zero sum games, min-max optimization and robust optimization.

If the polynomials are such that the set
$\mc{K}_0$ is infinite, then the convergence of
Algorithm~\ref{alg:pop:kt} is not theoretically guaranteed.
For future work, the following questions are important and interesting.

\begin{question}
When $F,g,h$ are generic,
what is an accurate (or sharp) upper bound for the number of iterations
required by Algorithm~\ref{alg:pop:kt} to terminate?
What is the complexity of Algorithm~\ref{alg:pop:kt}
for generic $F,g,h$?
\end{question}

Theorem~\ref{thm:iter:bound} gives an upper bound
for the number of iterations. However, the bound
given in \reff{upbd:M} is certainly not sharp.
Beyond the number of iterations, the complexity of
solving the polynomial optimization problems
\reff{min:F(x,y*)}, \reff{max:F(x*,y)} and \reff{minF:KT:uv}
is another concern. For efficient computational performance,
Algorithms~\ref{alg:LaSDP:xy}, \ref{alg:min:F(xy*)}, \ref{alg:maxF(y):Y}
are applied to solve them. However, their complexities are mostly open,
to the best of the authors' knowledge.
We remark that Algorithms~\ref{alg:LaSDP:xy}, \ref{alg:min:F(xy*)}, \ref{alg:maxF(y):Y}
are based on the tight relaxation method in \cite{LagExp},
instead of the classical Lasserre relaxation method in \cite{Las01}.
For the method in \cite{Las01}, there is no complexity result for the worst cases, i.e.,
there exist instances of polynomial optimization such that
the method in \cite{Las01} does not terminate within finitely many steps.
The method in \cite{LagExp} always terminates within finitely many steps,
under nonsingularity assumptions on constraints,
while its complexity is currently unknown.

The finite convergence of Algorithm~\ref{alg:pop:kt} is guaranteed
if the set $\mc{K}_0 \setminus \mc{S}_a$ is finite.
If it is infinite, it may or may not have finite convergence.
If it does not, how can we get a saddle point?
The following question is mostly open for the authors.

\begin{question}
For polynomials $F,g,h$ such that
the set $\mc{K}_0 \setminus \mc{S}_a$ is not finite,
how can we compute a saddle point if it exists?
Or how can we detect its nonexistence if it does not exist?
\end{question}

When $X,Y$ are nonempty compact convex sets and the function
$F$ is convex-concave, there always exists a saddle point \cite[\S2.6]{BNO03}.
However, if one of $X,Y$ is nonconvex or if $F$ is not convex-concave,
a saddle point may, or may not, exist.
This is the case even if $F$ is a polynomial
and $X,Y$ are semialgebraic sets.
The existence and nonexistence of saddle points for SPPPs
are shown in various examples in Section~\ref{sc:num}.
However, there is a new interesting property for SPPPs.
We can write the objective polynomial $F(x,y)$ as
\[
F(x,y) \, = \, [x]_{d_1}^T G [y]_{d_2},
\]
for degrees $d_1, d_2 > 0$ and a matrix $G$
(see Section~\ref{sc:pre} for the notation $[x]_{d_1}$ and $[y]_{d_2}$).
Consider new variables $u:=[x]_{d_1}$, $v:=[y]_{d_2}$ and the new sets
\[
\mc{X} = \{ [x]_{d_1}:\, x \in X \}, \quad
\mc{Y} = \{ [y]_{d_2}:\, y \in Y \}.
\]
A convex moment relaxation for the SPPP is to find
$(u^*,v^*) \in \cv{\mc{X}} \times \cv{\mc{Y}}$ such that
\[
(u^*)^TGv \, \le \, (u^*)^T G (v^*) \, \le \, u^T G v^*
\]
for all $u \in \cv{\mc{X}}$, $v \in \cv{\mc{Y}}$.
(The notation $\cv{T}$ denotes the convex hull of $T$.)
When $X,Y$ are nonempty compact sets, the above $(u^*, v^*)$
always exists, because $u^TGv$ is bilinear in $(u,v)$
and $\cv{\mc{X}},\cv{\mc{Y}}$ are compact convex sets.
In particular, if such $u^*$ is an extreme point of
$\cv{\mc{X}}$ and $v^*$ is an extreme point of $\cv{\mc{Y}}$,
say, $u^*=[a]_{d_1}$ and $v^*=[b]_{d_2}$ for $a \in X, b\in Y$,
then $(a,b)$ must be a saddle point of the original SPPP.
If there is no saddle point $(u^*,v^*)$ such that
$u^*,v^*$ are both extreme,
the original SPPP does not have saddle points.
We refer to \cite{LarLas12} for related work about this kind of problems.

\end{document}